\documentclass{amsart}

\usepackage{amssymb}

\usepackage{amscd}

\usepackage{graphicx,psfrag,epsfig,multirow}
\usepackage{amssymb,amsmath,amscd,amsthm,verbatim}
\usepackage{amssymb,amsmath,amscd,amsthm}
\usepackage{amsfonts,epsfig,latexsym,graphicx,amssymb, pict2e}

\newcommand{\cN}{\mathcal{N}}

\newcommand{\dt}{\delta}

\newcommand{\Lb}{\Lambda}

\newcommand{\T}{\mathbb{T}}

\renewcommand{\date}[1]{\begin{center}#1\end{center}}

\newtheorem{Thm}{Theorem}
\newtheorem{Def}{Definition}
\newtheorem{Lm}{Lemma}
\newtheorem{Prop}[Lm]{Proposition}
\newtheorem{Rem}{Remark}
\newtheorem{Cor}[Lm]{Corollary}

\def\bdef{\begin{Def}}
\def\endef{\end{Def}}
\def\bthm{\begin{Thm}}
\def\ethm{\end{Thm}}
\def\bprop{\begin{Prop}}
\def\enprop{\end{Prop}}
\def\blm{\begin{Lm}}
\def\elm{\end{Lm}}
\def\bcor{\begin{Cor}}
\def\ecor{\end{Cor}}
\def\brm{\begin{Rem}}
\def\erm{\end{Rem}}
\def\bfig{\begin{picture}}
\def\efig{\end{picture}}
\def\beq{\begin{eqnarray}}
\def\eneq{\end{eqnarray}}
\def\beal{\begin{aligned}}
\def\enal{\end{aligned}}

\title{On stochastic sea 
of the standard map}
\author{Anton Gorodetski}

\address{Department of Mathematics, University of California, Irvine CA 92697, USA}

\email{asgor@math.uci.edu}

\thanks{This work was supported in part by NSF grant DMS--0901627.}

\makeatletter

\def\keywords#1{{\def\@thefnmark{\relax}\@footnotetext{#1}}}

\let\subjclass\keywords

\makeatother

\begin{document}

\begin{abstract}
Consider a generic one-parameter unfolding of a homoclinic tangency of an area preserving surface diffeomorphism. We show that for many parameters (residual subset in an open set approaching the critical value) the corresponding diffeomorphism has  a transitive invariant set $\Omega$ of full Hausdorff dimension. The set $\Omega$ is a topological limit of hyperbolic sets and is accumulated by elliptic islands.

As an application we prove that stochastic sea
 of the standard map has full Hausdorff dimension for sufficiently large topologically generic parameters.
\end{abstract}

\maketitle



\section{Introduction and Main Results}

Here we prove that {\it stochastic sea
 of the Taylor-Chirikov standard map (i.e. the set of orbits with non-zero Lyapunov exponents) has full Hausdorff dimension for large  topologically generic parameters}. In order to do so we show that {\it a perturbation of an area preserving  diffeomorphism with a homoclinic tangency has hyperbolic invariant sets of almost full Hausdorff dimension.}

\subsection{Standard map}\label{s.std}

The simplest and most famous symplectic system with highly non-trivial dynamics is the
Taylor-Chirikov standard map of the two--dimensional torus $\T^2$,
given by
\begin{equation}\label{e.standardmap}
f_k(x,y)=(x+y+k\sin (2\pi x), y+k\sin (2\pi x))\ \text{\rm mod} \
\mathbb{Z}^2.
\end{equation}
This family is related to numerous physical problems, see for example \cite{ch1}, \cite{i}, \cite{ss}. For $k=0$ the map $f_0(x,y)=(x+y, y)$ is completely integrable, and leaves the circles $\{y=\text{\rm const}\}$ invariant. Due to KAM theory, after a perturbation large part of the torus is still foliated by invariant smooth circles, but lots of other phenomena appears: splitting of separatrices \cite{GL}, invariant Cantor sets \cite{MMP}, and many others \cite{BL}, \cite{L}, \cite{Go}. Computer generated pictures show that chaotic part of the phase space (orbits with positive Lyapunov exponents) also form a subset of positive measure, but this was never rigorously justified. Due to Pesin's theory \cite{P} this is equivalent to positivity of the metric entropy.

{\bf Main Question.} ({\it Sinai} \cite{Sin}) {\it Is the metric entropy
of $f_k$ positive for some values of $k$? for positive measure of
values of $k$? for all non-zero values of $k$?}

 A stronger version of this question is a famous conjecture which claims  that the limit density at infinity of the set of parameters $k$ for which the standard map $f_k:\mathbb{T}^2\to \mathbb{T}^2$ is ergodic (and therefore has no elliptic islands) and non-uniformly hyperbolic with respect to Lebesgue measure is equal to one. At the same time, it is known that the set of parameters $k$ with this property (if non-empty) must be nowhere dense in a neighborhood of infinity \cite{Du3}.

In a more general way, one can ask (see \cite{bu}, \cite{X2}) whether an analytic symplectic map of a connected manifold can have coexisting chaotic
component of positive measure and the Kolmogorov-Arnold-Moser (KAM)
tori. There are $C^{\infty}$ examples with this type of {\it mixed behavior} \cite{bu}, \cite{do}, \cite{li}, \cite{pr}, \cite{w}, but the rigorous proof of existence of mixed behavior for a real analytic map is still missing.

 Our main result claims, roughly speaking, that stochastic sea 
 of the standard map has full
Hausdorff dimension for large topologically generic parameters. 

\bthm\label{t.1} There exists $k_0>0$ and a residual set
$\mathcal{R}\in [k_0, +\infty)$ such that for every $k\in
\mathcal{R}$ there exists an infinite sequence of transitive locally
maximal hyperbolic sets of the map $f_k$ \beq\label{e.sequence}
\Lambda_k^{(0)}\subseteq \Lambda_k^{(1)} \subseteq
\Lambda_k^{(2)}\subseteq \ldots \subseteq\Lambda_k^{(n)}\subseteq
\ldots\eneq that has the following properties:

1. The family of sets $\{\Lambda_k^{(0)}\}_{k\ge k_0}$ is  dynamically increasing:   for small
$\varepsilon >0$, $\Lambda^{(0)}_{k+\varepsilon}$ contains the
continuation of $\Lambda_k^{(0)}$ at parameter $k+\varepsilon$;

2. The set $\Lb_k^{(0)}$ is
$\dt_k$-dense in $\T^2$ for $\delta_k=\frac{4}{k^{1/3}}$;

3. Hausdorff dimension $\text{\rm dim}_H\Lambda_k^{(n)} \to 2$ as
$n\to \infty$;

4. $\Omega_k=\overline{\cup_{n\in \mathbb{N}}  \Lambda_k^{(n)}}$
is a transitive invariant set of the map $f_k$, and $\text{\rm dim}_H \Omega_k=2$;

5. for any $x\in \Omega_k, k\in \mathcal{R}$, and any
$\varepsilon>0$ Hausdorff dimension
$$ \text{\rm
dim}_HB_{\varepsilon}(x)\cap \Omega_k=\text{\rm dim}_H \Omega_k=2,
$$
where $B_{\varepsilon}(x)$ is an open ball of radius $\varepsilon$
centered at $x$;

6. Each point of $\Omega_k$ is an accumulation point of elliptic islands of the map $f_k$.
\ethm

The family of hyperbolic sets $\{\Lambda_k^{(0)}\}$ that satisfies properties 1. and 2. was  constructed by Duarte in \cite{Du3}. He also showed that $\text{\rm dim}_H \Lambda_k^{(0)}\to 2$ as $k\to \infty$, and that for topologically generic parameters the set $\Lambda_k^{(0)}$ is accumulated by elliptic islands.

For an open set of parameters our construction provides invariant
hyperbolic sets of Hausdorff dimension arbitrarily close to 2.
\bthm\label{t.2} There exists $k_0>0$ such that for any $\xi>0$
there exists an open and dense subset $U\in [k_0, +\infty)$ such
that for every $k\in U$ the map $f_k$ has an invariant locally maximal hyperbolic
set of Hausdorff dimension greater than $2-\xi$ which is also $\delta_k$-dense in $\mathbb{T}^2$ for $\delta_k=\frac{4}{k^{1/3}}$. \ethm

Notice that these results give a partial explanation of the
difficulties that we encounter studying the standard family. Indeed,
one of the possible approaches is to consider an invariant
hyperbolic set in the stochastic layer and to try to extend the
hyperbolic behavior to a larger part of the phase space through
homoclinic bifurcations. Unavoidably Newhouse domains (see
\cite{N}, \cite{n2}, \cite{N3}, \cite{R} for dissipative case, and \cite{Du1},
\cite{Du2}, \cite{Du4}, \cite{gs2} for the conservative case) associated with
absence of hyperbolicity appear after small change of the parameter.
If the Hausdorff dimension of the initial hyperbolic set is less
than one, then the measure of the set of parameters that correspond
to Newhouse domains is small and has zero density at the critical
value, see \cite{NP}, \cite{PT1}. For the case when the Hausdorff
dimension of the hyperbolic set is slightly bigger than one, similar
result was recently obtained by Palis and Yoccoz \cite{PY}, and the
proof is astonishingly involved. They also conjectured that
analogous property holds for an initial hyperbolic set of any
Hausdorff dimension, but the proof would require even more technical
and complicated considerations.\footnote{Here is a citation from Palis and Yoccoz
\cite{PY}:
{\it ``Of course, we expect the same to be true for all cases $0 <
\text{\rm dim}_H(\Lambda) < 2$. For that, it seems to us that our
methods need to be considerably sharpened: we have to study deeper
the dynamical recurrence of points near tangencies of higher order
(cubic, quartic, ...) between stable and unstable curves. We also
hope that the ideas introduced in the present paper might be useful
in broader contexts. In the horizon lies the famous question whether
for the standard family of area preserving maps one can find sets of
positive Lebesgue probability in parameter space such that the
corresponding maps display non-zero Lyapunov exponents in sets of
positive Lebesgue probability in phase space.''}}
Theorems \ref{t.1} and \ref{t.2} show that in order to understand
the properties of the stochastic sea of the standard map using this approach one has to
face these difficulties.

\subsection{Hyperbolic sets of large Hausdorff dimension}\label{s.hd}

Several famous long standing conjectures (including Main Question above) discuss the measure of certain invariant sets
of some dynamical systems.
Any set of positive Lebesgue  measure  has Hausdorff dimension which is equal to
the dimension of the ambient manifold. Therefore it is reasonable to ask first whether
those invariant sets indeed have full Hausdorff dimension.

In dissipative setting Downarowicz and Newhouse \cite{DN} proved  that there is a residual subset $\mathcal{R}$ of
the space of $C^r$-diffeomorphisms of a compact two dimensional manifold $M$ such that if
$f\in \mathcal{R}$ and $f$ has a homoclinic tangency, then $f$ has compact invariant topologically
transitive sets of Hausdorff dimension two. Their methods use essentially perturbative technics (see \cite{gst}) and therefore cannot be
generalized to the finite parameter families.

In conservative setting  Newhouse
\cite{N4} proved that in $\text{\rm Diff}^{\,1}(M^2, Leb)$ there is a
residual subset of maps such that every homoclinic class\footnote{See Definition \ref{d.hc} below.}
for each of
those maps has Hausdorff dimension 2. Later Arnaud, Bonatti and
 Crovisier \cite{BC}, \cite{ABC} essentially improved that result and showed that in the space of $C^1$ symplectic
  maps the residual subset consists of the transitive  maps
that have only one homoclinic class (the whole manifold).
Notice that due to KAM theory the low smoothness in that work is essential.

Here we show that
{\it a generic one parameter area-preserving homoclinic bifurcation always give birth to a compact invariant topologically transitive set of Hausdorff dimension two.} This set is the closure of the union of
a countable sequence of hyperbolic sets of Hausdorff dimension arbitrary close to two.

\subsubsection{The area preserving Henon family}

First of all we consider area preserving Henon family (\ref{hf}). For $a=-1$ this map has
a degenerate fixed point at $(x,y)=(-1,1)$. We construct invariant hyperbolic sets of large
Hausdorff dimension for $a$ slightly larger than $-1$ near this fixed point. Later we use
the renormalization results to reduce the case of a generic unfolding of an area preserving
surface diffeomorphism with a homoclinic tangency to this construction.

\bthm\label{t.henonfamily} Consider the family of area preserving
Henon maps \beq\label{hf} H_a:
\begin{pmatrix}
x \\
y \\
\end{pmatrix}\mapsto \begin{pmatrix}
y \\
-x+a-y^2 \\
\end{pmatrix}.\eneq
There is a (piecewise continuous) family of sets $\Lambda_a$,
$a\in[-1, -1+\varepsilon]$ for some $\varepsilon>0$,  such that
the following properties hold.

1. The set $\Lambda_a$ is a locally maximal hyperbolic set of the
map $H_a$;

2. The set $\Lambda_a$ contains a saddle fixed point of the map
$H_a$;

3. The set $\Lambda_a$ has an open and closed (in $\Lambda_a$)
subset $\widetilde{\Lambda}_a$ such that the first return map for
$\widetilde{\Lambda}_a$ is a two-component Smale horseshoe; 

4. Hausdorff dimension $dim_H\widetilde{\Lambda}_a\to 2$ as $a\to
-1$.
 \ethm

 Theorem \ref{t.henonfamily} should be considered as an improvement of Lemma A from \cite{Du4}, where  Duarte proves that area preserving Henon maps have hyperbolic sets of large ``left-right thickness" (see \cite{Du4, Mo} for a definition) for values of $a$ slightly larger than $-1$.

The proof essentially uses the construction from  \cite{Du1}, \cite{Du2} that was used by Duarte to study conservative Newhouse phenomena, and   results regarding the splitting of separatrices for Henon family from \cite{G1}, \cite{G2},
\cite{G3}, \cite{GSa}, \cite{BG} (see also  \cite{Ch}, where some numerical results are described). 

 A similar statement holds also for any generic one parameter unfolding of an extremal periodic point  (see
\cite{Du1} for a formal definition) as soon as the form of the splitting of separatrices can be established (see
\cite{G1,GL} for the relevant results on splitting of separatrices).

\subsubsection{Conservative homoclinic bifurcations and hyperbolic sets of large Hausdorff dimension}

  In dissipative case Newhouse \cite{N} showed that near every  surface
diffeomorphism with a homoclinic tangency there are open sets
(nowadays called {\it Newhouse domains}\,) of maps with persistence
homoclinic tangencies. Moreover, in these open sets there are
residual subsets of maps with infinitely many attracting periodic
orbits. Later Robinson \cite{R} showed that this result can be
formulated in terms of generic one parameter unfolding of a
homoclinic tangency.

 In area preserving case  Duarte \cite{Du1}, \cite{Du2}, \cite{Du4} showed that homoclinic tangencies also lead to similar
phenomena, the role of sinks is played by elliptic periodic points. Theorem \ref{t.maintheorem} below is a stronger version of the Duarte's
result: we can control the Hausdorff dimension of the hyperbolic
sets 
that appear in the construction.

In order to construct transitive invariant sets of full Hausdorff dimension we use the notion of a {\it homoclinic class.}

\bdef\label{d.hc} Let $P$ be a hyperbolic saddle of a diffeomorphism $f$. A homoclinic class $H(P,f)$ is a closure of the
union of all the transversal homoclinic points of $P$.
\endef
It is known that $H(P,f)$ is a transitive invariant set of $f$, see
\cite{N1}. Moreover, consider all basic sets (locally maximal
transitive hyperbolic sets) that contain the saddle $P$. A homoclinic
class $H(P,f)$ is a smallest closed invariant set that contains all
of them.

\bthm\label{t.maintheorem} Let $f_0\in \text{\rm Diff}^{\,\infty}(M^2, \text{Leb})$\footnote{We assume $C^{\infty}$-smoothness of diffeomorphisms here just for simplicity.
For the renormalization procedures and arguments used in the current proof it is enough to assume only $C^6$-smoothness (which is probably not optimal either), compare with \cite{Du4}. Since all the cases where we intend to apply this result (standard map, three body problems) are analytic, we are making no attempt to optimize the required class of smoothness.} 
 have an orbit $\mathcal{O}$ of quadratic homoclinic tangencies associated
to some hyperbolic fixed point $P_0$, and $\{f_{\mu}\}$ be a generic
unfolding of $f_0$ in $\text{\rm Diff}^{\,\infty}(M^2, \text{Leb})$.
Then for any
$\delta>0$ there  is an open set $\mathcal{U}\subseteq \Bbb{R}^1$, $0\in
\overline{\mathcal{U}}$, such that the following holds:

$(1)$ for every $\mu\in \mathcal{U}$ the map $f_{\mu}$ has a basic
set $\Delta_{\mu}$ that contains the unique fixed point $P_{\mu}$
near $P_0$,  exhibits persistent homoclinic tangencies, and
Hausdorff dimension
$$\text{\rm dim}_H\Delta_{\mu}>2-\delta;$$

$(2)$ there is a dense subset $\mathcal{D}\subseteq \mathcal{U}$
such that for every $\mu\in \mathcal{D}$ the map $f_\mu$ has a
homoclinic tangency of the fixed point $P_{\mu}$;

$(3)$ there is a residual subset $\mathcal{R}\subseteq \mathcal{U}$
such that for every $\mu\in \mathcal{R}$

$ \qquad (3.1)$ the homoclinic class $H(P_{\mu}, f_{\mu})$
is accumulated by $f_{\mu}$'s generic elliptic
points,

$ \qquad (3.2)$ the homoclinic class $H(P_{\mu}, f_{\mu})$ contains
hyperbolic sets of Hausdorff dimension arbitrary close to 2;
in particular,  $\text{\rm dim}_HH(P_{\mu},f_{\mu})=2$,

$ \qquad (3.3)$ $\text{\rm dim}_H\{x\in H(P_{\mu},f_{\mu})\ |\, P_{\mu}\in \omega(x)\cap \alpha(x)\}=2$.
 \ethm

As usual, when we have a property that holds for a topologically generic parameter values, it is interesting to find out whether it holds for almost every parameter value, or with positive probability (i.e. for a positive measure set of parameters). For dissipative Newhouse phenomena see \cite{ly}, \cite{GKsinks}, \cite{GHK} for some results in this direction. In the context  of Theorem \ref{t.maintheorem} this leads to the following questions.

{\bf Problem 1. }  {\it Under conditions of Theorem \ref{t.maintheorem}, what is the measure of the parameters $\mu$ such that $\text{\rm dim}_HH(P_{\mu}, f_{\mu})=2$? Such that $\text{\rm dim}_HH(P_{\mu}, f_{\mu})>2-\varepsilon$? }

Also, for some applications (see \cite{GK}) it would be useful to improve the item (3.3) of Theorem \ref{t.maintheorem}.

{\bf Problem 2.} {\it Under conditions of Theorem \ref{t.maintheorem}, prove that for every $\mu\in \mathcal{R}$ the set of points with dense orbits in the homoclinic class $H(P_{\mu},f_{\mu})$ has full Hausdorff dimension.
}

Initially our interest in the conservative Newhouse phenomena was
motivated by the fact that it appears in the three body problem.
Namely, let us try to understand the structure of the set of
oscillatory motions (a planet approaches infinity always returning
to a bounded domain) in a Sitnikov problem \cite{A, sit}. It is a special case
of the restricted three body problem where the two primaries with
equal masses are moving in an elliptic orbits of the two body
problem, and the infinitesimal mass is moving on the straight line
orthogonal to the plane of motion of the primaries which passes
through the center of mass. The eccentricity of the orbits of
primaries is a parameter. After some change of coordinates (McGehee
transformation \cite{McG}) the infinity can be considered as a degenerate
saddle with smooth invariant manifolds that correspond to parabolic
motions (the orbit tends to infinity with zero limit velocity).
Stable and unstable manifolds coincide in the case of circular (parameter is equal to zero)
Sitnikov problem. It is known that
  for non-zero eccentricity invariant manifolds
have a point of transverse intersection \cite{GP}, \cite{DH}, \cite{moser}. This leads to the existence
of homoclinic tangencies and appearance of all phenomena that
can be encountered in the conservative
 homoclinic bifurcations. Similar statement holds for the planar circular
 restricted three body problem. The existence of transversal
 homoclinic points in the latter case was established in
 \cite{LS}, \cite{X}. The farther development of this approach is a subject of our current joint project with V.Kaloshin, see \cite{GK} for some preliminary results.

The structure of the paper is the following. In Section \ref{s.leftrightCantor} we remind the definitions of lateral (left- and right-) thickness of a Cantor set and show how Hausdorff dimension of a dynamically defined Cantor set can be estimated via its lateral thicknesses. In Section \ref{s.distest} Duarte's Distortion Theorem that allows to estimate thickness of a non-linear horseshoe is improved to cover a larger class of horseshoes. In Section \ref{s.splitting} we discuss the results by Gelfreich and Sauzin \cite{GSa} on splitting of separatrices in the area preserving Henon family, and then in Section \ref{s.thickhorseshoe} we apply those results together with results of Sections \ref{s.leftrightCantor} and \ref{s.distest} to show how a horseshoe of large Hausdorff dimension appears (i.e. prove Theorem \ref{t.henonfamily}). In Section \ref{s.consproof} we consider conservative homoclinic bifurcations and construct hyperbolic sets of large Hausdorff dimension (and prove Theorem \ref{t.maintheorem}), and, finally, in Section \ref{s.standardproof} we use this result to show that stochastic layer of the standard map has full Hausdorff dimension for many values of the parameter (i.e. prove Theorems \ref{t.1} and \ref{t.2}).

\subsubsection*{Acknowledgments.} {I would like to thank V.~Kaloshin for numerous insights, advices, and support, and to emphasize that originally this paper was motivated by our joint project on Hausdorff dimension of oscillatory motions in three body problems. Also I thank V.~Gelfreich for the patience that he exercised answering my questions regarding his results on splitting of separatrices, and P.~Duarte, T.~Fisher, S.~Newhouse, D.~Saari, and D.~Turaev for useful discussions.}

\section{Left-right thickness and Hausdorff dimension of Cantor sets}\label{s.leftrightCantor}

It is known that a Cantor set of large thickness must have large Hausdorff dimension \cite{PT}. In our construction we will encounter a Cantor set of small thickness. Nevertheless, we are still able to estimate Hausdorff dimension of the constructed Cantor sets. Namely, following the ideas by Moreira \cite{Mo} and Duarte \cite{Du2}, we use lateral (left- and right-) thickness of a Cantor set, and we will see that the Cantor sets in our construction have one of the lateral thicknesses large and another one small (but controlled). In this section we show how to estimate Hausdorff dimension of a Cantor set in this case.

\subsection{Dynamically defined Cantor sets}

Here we reproduce the definition of the left and right thickness from \cite{Du2} and \cite{Mo}. We will use
these one-sided thicknesses instead of the standard definition of thickness. See \cite{PT} for the usual definition of the
thickness of a Cantor set.

Name {\it dynamically defined Cantor set} any pair $(K,\psi)$ such that $K\subseteq \Bbb{R}$ is a Cantor set and
$\psi:K\to K$ is a locally Lipschitz expanding map, topologically conjugated to some subshift of a finite type
of a Bernoulli shift $\sigma:\{0,1,\ldots,p\}^{\Bbb{N}}\to \{0,1,\ldots,p\}^{\Bbb{N}}$. For the sake of
simplicity, and because this is enough for our purpose, we will restrict ourselves to the case where $\psi$ is
conjugated to the full Bernoulli shift $\sigma: \{0,1\}^{\Bbb{N}}\to \{0,1\}^{\Bbb{N}}$. Also we will assume
that a {\it Markov partition} $\mathcal{P}=\{K_0,K_1\}$  of $(K,\psi)$ is given.  In our case this means that
the following properties are satisfied:

(1) $\mathcal{P}$ is a partition  of $K$ into disjoint union of two  Cantor subsets, $K=K_0\cup K_1$, $K_0\cap
K_1=\emptyset$;

(2) the restriction of $\psi$ to each $K_i$, $\psi|_{K_i}:K_i\to K$, is a strictly monotonous Lipschitz
expanding homeomorphism.

For a general definition of Markov partition see \cite{Mo}, \cite{PT}.

 Given a symbolic sequence $(a_0,\ldots, a_{n-1})\in \{0,1\}^n,$ denote
$$
K(a_0,\ldots, a_{n-1})=\cap_{i=0}^{n-1}\psi^{-i}(K_{a_{i}}),
$$
then the map $\psi^n:K(a_0,\ldots, a_{n-1})\to K$ is a Lipschitz expanding homeomorphism.

A bounded component of the complement $\Bbb{R}\backslash K$ is
called a {\it gap} of $K$. For a dynamically defined Cantor set
$(K,\psi)$ the gaps are ordered in the following way. Denote by
$\widehat{A}$ the convex hall of a subset $A\subseteq \Bbb{R}$. Then
the interval $\widehat{K}\backslash (\widehat{K_0}\cup
\widehat{K_1})$ is called a gap of order zero. A connected component
of
$$
\widehat{K}\backslash \cup_{(a_0,\ldots, a_{n-1})\in \{0,1\}^n}\widehat{K(a_0,\ldots, a_{n-1})}
$$
that is not a gap of order less than or equal to $n-1$ is called a gap of order $n$. It is straightforward to
check that every gap of $K$ is a gap of some finite order, and also that, given a gap $U=(x,y)$ of order $n$,
for every $0\le k\le n$ the open interval bounded by $\psi^k(x)$ and $\psi^k(y)$ is a gap of order $n-k$. \bdef
Given a gap $U$ of $K$ with order $n$, we denote by $L_U$, respectively $R_U$, the unique interval of the form
$\widehat{K(a_0,\ldots, a_{n-1})}$, with $(a_0,\ldots, a_{n-1})\in \{0,1\}^n$, that is left, respectively right,
adjacent to $U$. The greatest lower bounds
$$
\tau_L(K)=\inf\left\{\frac{|L_U|}{|U|}: U \text{ is a gap of }
 K\right\}
$$
$$
\tau_R(K)=\inf\left\{\frac{|R_U|}{|U|}: U \text{ is a gap of }
 K\right\}
$$
are respectively called the {\it left} and {\it right thickness} of
$K$. Similarly, the ratios
$$
\tau_L(\mathcal{P})=\frac{|L_{U_0}|}{|U_0|}\ \ \   \text{and}\ \ \
\tau_R(\mathcal{P})=\frac{|R_{U_0}|}{|U_0|},
$$
where $U_0$ is the unique gap of order zero, are called the left and
the right thickness of the Markov partition $\mathcal{P}$.
\endef

\begin{figure}
  \includegraphics{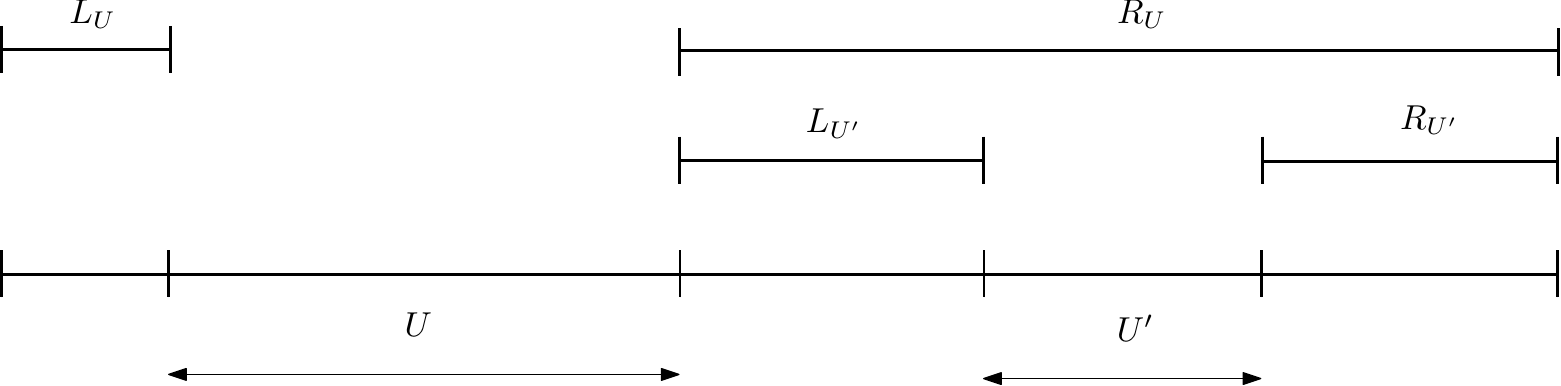}\\
  \caption{On definition of thickness}\label{fig1}
\end{figure}

Initially left- and right- thickness were introduced by Moreira\footnote{Original Moreira's definition is formally different, and can be used for any Cantor set, not necessarily for dynamically defined. We use the modification suggested by Duarte in \cite{Du2}. Lemma \ref{l.leftright} holds in either case, and Duarte's definition is more convenient for dynamically defined Cantor sets.} who proved the following generalization of the Newhouse's Gap Lemma \cite{N}.

\blm\label{l.leftright} {\rm (Left-right gap lemma, see \cite{Mo})} Let $(K^s, \psi^s)$, $(K^u, \psi^u)$ be
dynamically defined Cantor sets such that the intervals supporting $K^s$ and $K^u$ do intersect, $K^s$ (resp.
$K^u$) is not contained inside a gap of $K^u$ (resp. $K^s$).  If $\tau_L(K^s)\tau_R(K^u)>1$ and
$\tau_R(K^s)\tau_L(K^u)>1$, then both Cantor sets intersect, $K^s\cap K^u\ne \emptyset$.  \elm

\subsection{Large thickness implies large Hausdorff dimension}

Let us recall the definition of the Hausdorff dimension. Let
$K\subset \Bbb{R}$ be a Cantor set and $\mathcal{U}=\{U_i\}_{i\in
I}$ a finite covering of $K$ by open intervals in $\Bbb{R}$. We
define the diameter $\text{diam}(\mathcal{U})$ of $\mathcal{U}$ as
the maximum of $|U_i|$, $i\in I$, where $|U_i|$ denotes the length
of $U_i$. Define $H_{\alpha}(\mathcal{U})=\sum_{i\in
I}|U_i|^{\alpha}$. Then the {\it Hausdorff $\alpha$-measure} of $K$
is
$$
m_{\alpha}(K)=\lim_{\varepsilon\to 0}\left(\inf_{\mathcal{U} \text{
covers } K,\ \text{diam}(\mathcal{U})<\varepsilon}
H_{\alpha}(\mathcal{U})\right).
$$
It is not hard to see that there is a unique number, the {\it
Hausdorff dimension} of $K$, denoted by $\text{dim}_H(K)$, such that
for $\alpha<\text{dim}_H(K)$, $m_{\alpha}(K)=\infty$ and for
$\alpha>\text{dim}_H(K)$, $m_{\alpha}(K)=0$.

\bprop\label{dfromtau} Consider a Cantor set $K$, denote
$\tau_L=\tau_L(K)$ and  $\tau_R=\tau_R(K)$, and let $d$ be the
solution of the equation \beq\label{e.tau} \tau_L^d+\tau_R^d=(1+\tau_L+\tau_R)^d.
\eneq
Then $\dim_H(K)\ge d$.\enprop \brm One can consider Proposition
\ref{dfromtau} as a generalization of the Proposition 5 from Chapter
4.2 in \cite{PT}, where the relation between the usual thickness and
the Hausdorff dimension of a Cantor set was established. Indeed, if $\tau_L=\tau_R=\tau$ then Proposition \ref{dfromtau} implies that $\text{\rm dim}_H(K)\ge \frac{\log 2}{\log\left(2+\frac{1}{\tau}\right)}$, which is exactly the statement from \cite{PT}.  \erm
\begin{proof}[Proof of the Proposition \ref{dfromtau}.] We will need the following elementary Lemma.
 \blm\label{elementarylemma}
If $d\in(0,1)$ is a solution of the equation (\ref{e.tau}) then
\begin{align*}
    \min\left\{x^d+y^d\ |\ x\ge 0, y\ge 0, x+y\le 1, x\ge \tau_L(1-x-y), y\ge \tau_R(1-x-y)
\right\}=1.
\end{align*}
 \elm
\begin{proof}[Proof of Lemma \ref{elementarylemma}.]   The function $f(x,y)=x^d+y^d$ is concave, and takes value 1 at points $(0,1), (1, 0)$, and $\left(\frac{\tau_L}{1+\tau_L+\tau_R},
\frac{\tau_R}{1+\tau_L+\tau_R}\right)$. Therefore its minimum in the triangle with vertices at these points is equal to 1.
\end{proof}

 We show that $H_d(\mathcal{U})\ge
(\text{diam } K)^d$ for every finite open covering $\mathcal{U}$ of
$K$, which clearly implies the proposition. We can assume that
$\mathcal{U}$ is a covering with disjoint intervals. This is no
restriction because whenever two elements of $\mathcal{U}$ have
nonempty intersection we can replace them by their union, getting in
this way a new covering $\mathcal{V}$ such that $H_d(\mathcal{V})\le
H_d(\mathcal{U})$. Note that, since $\mathcal{U}$ is an open
covering of $K$, it covers all but finite number of gaps of $K$. Let
$U$, a gap of $K$, have minimal order among the gaps of $K$ which
are not covered by $\mathcal{U}$. Let $C^L$ and $C^R$ be that
bridges of $K$ at the boundary points of $U$.

  By construction there are $A^L$, $A^R\in \mathcal{U}$ such that $C^L\subset A^L$ and $C^R\subset A^R$. Take
  the convex hall $A$ of $A^L\cup A^R$. Then
  $$
|A^L|\ge |C^L|\ge \tau_L\cdot |U|\ge \tau_L(|A|-|A^L|-|A^R|)
  $$
and
$$
|A^R|\ge |C^R|\ge \tau_R\cdot |U|\ge \tau_R(|A|-|A^L|-|A^R|).
$$
Or, equivalently,
$$
\frac{|A^L|}{|A|}\ge
\tau_L\left(1-\frac{|A^L|}{|A|}-\frac{|A^R|}{|A|}\right)
$$
and
$$
\frac{|A^R|}{|A|}\ge
\tau_R\left(1-\frac{|A^L|}{|A|}-\frac{|A^R|}{|A|}\right).
$$
Lemma \ref{elementarylemma} now implies that
$$
\left(\frac{|A^L|}{|A|}\right)^d+\left(\frac{|A^R|}{|A|}\right)^d\ge
1,
$$
and $|A^L|^d+|A^R|^d\ge |A|^d$. This means that the covering
$\mathcal{U}_{\,1}$ of $K$ obtained by replacing $A^L$ and $A^R$ by
$A$ in $\mathcal{U}$ is such that $H_d(\mathcal{U}_{\,1})\le
H_d(\mathcal{U})$. Repeating the argument we eventually construct
$\mathcal{U}_{\,k}$, a covering of the convex hall of $K$ with
$H_d(\mathcal{U}_{\,k})\le H_d(\mathcal{U})$. Since we must have
$H_d(\mathcal{U}_{\,k})\ge (\text{diam }K)^d$, this finishes the proof.\end{proof}

Proposition \ref{dfromtau} can be used to find an explicit estimate
of the Hausdorff dimension via one-sided thicknesses. In particular,
when one of the one-sided thicknesses is very large and another one
is small, the following Proposition gives an estimate that is good
enough for our purposes.

\bprop\label{logestimate} Denote by $\tau_L$ and $\tau_R$ the left
and right thicknesses of the Cantor set $K\subset \Bbb{R}$. Then
$$
\text{dim}_HK> \max \left(\
\frac{\log\left(1+\frac{\tau_R}{1+\tau_L}\right)}{\log\left(1+\frac{1+\tau_R}{\tau_L}\right)},\
\ \ \
\frac{\log\left(1+\frac{\tau_L}{1+\tau_R}\right)}{\log\left(1+\frac{1+\tau_L}{\tau_R}\right)
} \ \right).
$$
\enprop

 \begin{proof}[Proof of Proposition \ref{logestimate}.] We will use the following Lemma.
 \blm\label{inequality}
Assume that for some $x,y>0$, $x+y<1$, and some $d_1, d_2\in (0,1)$
the following relations hold:
$$
y=(1-x)^{\frac{1}{d_1}}, \ \ \ \ \ x^{d_2}+y^{d_2}=1.
$$
Then $d_2>d_1$.
 \elm
\begin{proof}[Proof of Lemma \ref{inequality}.]  Indeed, $(1-x)^{\frac{1}{d_1}}=y=(1-x^{d_2})^{\frac{1}{d_2}}<(1-x)^{\frac{1}{d_2}}$, so due
 to our choice of $x$ we have  $d_2>d_1$. \end{proof}

 Let us apply Lemma \ref{inequality} to $x=\frac{\tau_L}{1+\tau_L+\tau_R}$ and
 $y=\frac{\tau_R}{1+\tau_L+\tau_R}$. If $y=(1-x)^{\frac{1}{d_1}}, x^{d_2}+y^{d_2}=1$ for some $d_1, d_2\in (0,1)$
 then by Proposition \ref{dfromtau} we have
 $$\text{dim}_HK\ge d_2>d_1=\frac{\log (1-x)}{\log y}=\frac{\log\left(1-\frac{\tau_L}{1+\tau_L+\tau_R}\right)}
 {\log\left(\frac{\tau_R}{1+\tau_L+\tau_R}\right)}=\frac{\log\left(1+\frac{\tau_L}{1+\tau_R}\right)}
 {\log\left(1+\frac{1+\tau_L}{\tau_R}\right) }.$$
In a similar way one can show that
$\text{dim}_HK>\frac{\log\left(1+\frac{\tau_R}{1+\tau_L}\right)}{\log\left(1+\frac{1+\tau_R}{\tau_L}\right)}$.
\end{proof}

\brm\label{remark5} Assume that $\tau_R\sim \frac{1}{\lambda-1}$,
$\tau_L\sim (\lambda-1)^{\nu}$. Then $$ \lim_{\lambda\to
1+0}\frac{\log\left(1+\frac{\tau_R}{1+\tau_L}\right)}{\log\left(1+\frac{1+\tau_R}{\tau_L}\right)}=\lim_{\lambda\to
1+0}\frac{\log\left(1+\frac{\frac{1}{\lambda-1}}{1+(\lambda-1)^{\nu}}\right)}{\log\left(1+\frac{1+\frac{1}
{\lambda-1}}{(\lambda-1)^{\nu}}\right)}=\frac{1}{1+\nu}.
$$
So if $\nu$ is small enough and $\lambda$ is close to one, then
$\text{{\rm dim}}_HK$ is close to 1.
\erm

\section{Nonlinear horseshoes and distortion estimates}\label{s.distest}

Here we describe the way to estimate lateral thicknesses of a non-linear horseshoe. We follow the approach from \cite{Du2}, but with some modifications; since our goal is to construct horseshoes with large (close to 2) Hausdorff dimension, we have to deal with a larger class of horseshoes than the one considered in \cite{Du2}.

\subsection{Non-linear horseshoes and their Markov partitions.}

\bdef\label{classf} Define $\mathcal{F}$ to be the set of all maps $f:\mathbf{S}_0\cup\mathbf{S}_1\to \Bbb{R}^2$
such that:

(1) $\mathbf{S}_0$, $\mathbf{S}_1\subset \mathbb{R}^2$ are compact sets, diffeomorphic to rectangles, with non-empty interior;

(2) $f$ is a map of class $C^2$, in a neighborhood of $\mathbf{S}_0\cup\mathbf{S}_1$, mapping this compact set
diffeomorphically onto its image $f(\mathbf{S}_0)\cup f(\mathbf{S}_1)$;

(3) the maximal invariant set $\Lambda(f)=\cap_{n\in \Bbb{Z}}f^{-n}(\mathbf{S}_0\cup\mathbf{S}_1)$ is a
hyperbolic basic set conjugated to the topological Bernoulli shift $\sigma:\{0,1\}^{\Bbb{Z}}\to \{0,1\}^{\Bbb{Z}}$;

(4) $\mathcal{P}=\{\mathbf{S}_0, \mathbf{S}_1\}$ is a Markov partition for $f:\Lambda(f)\to \Lambda(f)$, in
particular, $f$ has two fixed points, $\mathbf{P}_0\in \mathbf{S}_0$ and $\mathbf{P}_1\in \mathbf{S}_1$, whose
stable and unstable manifolds contain the boundaries of $\mathbf{S}_0$ and $\mathbf{S}_1$;

(5) both fixed points $\mathbf{P}_0$ and $\mathbf{P}_1$ have positive eigenvalues.
\endef

\begin{figure}
  \includegraphics[width=0.7\textwidth]{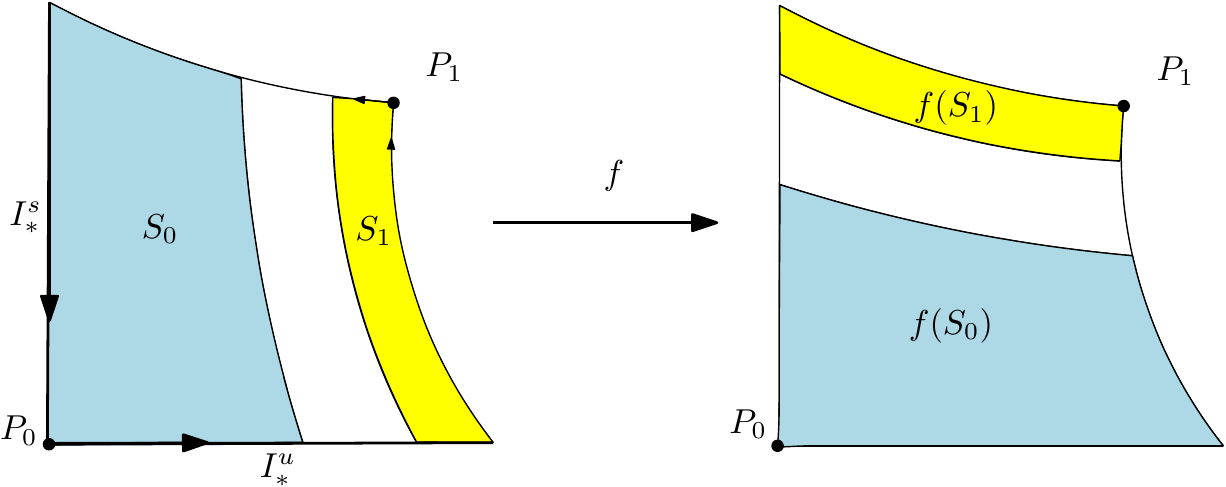}\\
  \caption{A nonlinear horseshoe $f\in \mathfrak{F}$}\label{fig3}
\end{figure}

The action of $f$ and $f^{-1}$ respectively on the stable, and unstable, foliation of $\Lambda$,
$$
 \mathcal{F}^s=\{ \text{\rm connected comp. of } W^s(\Lambda)\cap (\mathbf{S}_0\cup\mathbf{S}_1)\},
$$
$$
 \mathcal{F}^u=\{ \text{\rm connected comp. of } W^u(\Lambda)\cap (f(\mathbf{S}_0)\cup f(\mathbf{S}_1))\},
$$
can be described in the following way. Define
$$
I_*^s=W^s_{loc}(\mathbf{P}_0)\cap \mathbf{S}_0\ \ \ \  \ \text{\rm and} \ \ \ \
I_*^u=W^u_{loc}(\mathbf{P}_0)\cap f(\mathbf{S}_0).
$$
$I_*^s$ and $I_*^u$ are stable and unstable leaves of $\Lambda$ respectively transversal to the foliation
$\mathcal{F}^u$ and $\mathcal{F}^s$. Then the Cantor sets
$$
K^s=\Lambda\cap I_*^u\ \ \ \ \ \  \text{\rm and} \ \ \ \ \  K^u=\Lambda\cap I_*^s,
$$
can be identified with the set of stable leaves of $\mathcal{F}^s$, respectively unstable leaves of
$\mathcal{F}^u$. Define the projections $\pi_s:\Lambda\to K^s$ and $\pi_u:\Lambda\to K^u$ in the obvious way:
$\pi_s(P)$ is the unique point in $W^s_{loc}(P)\cap I_*^u$, and similarly $\pi_u(P)$ is the unique point in
$W^u_{loc}(P)\cap I_*^s$. The maps $\psi^s:K^s\to K^s$ and $\psi^u:K^u\to K^u$,
$$
\psi^s=\pi_s\circ f \ \ \ \ \  \text{\rm and} \ \ \ \ \ \psi^u=\pi_u\circ f^{-1},
$$
describe the action of $f$, respectively $f^{-1}$, on stable, respectively unstable leaves of $\Lambda$. The
pairs $(K^s, \psi^s)$ and $(K^u, \psi^u)$ are dynamically defined Cantor sets, topologically conjugated to the
Bernoulli shift $\sigma:\{0,1\}^{\Bbb{N}}\to \{0,1\}^{\Bbb{N}}$, with Markov partitions
$\mathcal{P}^u=\{I^u_*\cap \mathbf{S}_0, I^u_*\cap \mathbf{S}_1\}$ and $\mathcal{P}^s=\{I^s_*\cap
f(\mathbf{S}_0), I^s_*\cap f(\mathbf{S}_1)\}$.

\subsection{Distortion of a dynamically defined Cantor set.}

\bdef Given a Lipschitz expanding map $g:J\to \Bbb{R}$, defined on some subset $J\subset \Bbb{R}$, we define
distortion of $g$ on $J$ in the following way:
$$
\text{\rm Dist}(g, J)=\sup_{x,y,z\in J} \log
\left\{\frac{|g(y)-g(x)|}{|g(z)-g(x)|}\frac{|z-x|}{|y-x|}\right\}\in [0,+\infty],
$$
where the sup is taken over all $x,y,z\in J$ such that $z\ne x$ and $y\ne x$; due to injectivity of $g$ this
implies that $g(z)\ne g(x)$ and $g(y)\ne g(x)$.
\endef
Reversing the roles of $y$ and $z$ we see that the distortion is always greater than or equal to $\log 1=0$. If
$\text{Dist}(g, J)=c$, then for all $x,y,z\in J$ with $z\ne x$ and $y\ne x$ we have
$$
e^{-c}\frac{|y-x|}{|z-x|}\le \frac{|g(y)-g(x)|}{|g(z)-g(x)|}\le e^c \frac{|y-x|}{|z-x|}.
$$
\bdef The distortion of a dynamically defined Cantor set $(K, \psi)$ is defined as
$$
\text{\rm Dist}_{\psi}(K)=\sup \text{\rm Dist}(\psi^n, K(a_0, \ldots, a_{n-1}))
$$
taken over all sequences $(a_0, \ldots, a_{n-1})\in \{0,1\}^n$.
\endef
\blm[see \cite{PT}, \cite{Du2}]\label{dist} Let $(K, \psi)$ be a dynamically defined Cantor set with a
Markov partition $\mathcal{P}$ and distortion $\text{\rm Dist}_{\psi}(K)=c$. Then
$$
e^{-c}\tau_L(\mathcal{P})\le \tau_L(K)\le e^c\tau_L(\mathcal{P}), \ \ \ \ e^{-c}\tau_R(\mathcal{P})\le
\tau_R(K)\le e^c\tau_R(\mathcal{P}).
$$
\elm

\subsection{Duarte's Distortion Theorem }

For $C^*=2$ the following Definition coincides with Definition 4 from \cite{Du2}.

\begin{Def}\label{classofmaps}
Given positive constants $C^*$ along with small $\varepsilon$ and $\gamma$, define $\mathcal{F}(C^*,
\varepsilon, \gamma)$ to be the class of all maps $f:\mathbf{S}_0\cup\mathbf{S}_1\to \Bbb{R}^2, \ f\in \mathcal{F},$ such that:

(1)\ \ \ $\text{\rm diam\,}(\mathbf{S}_0\cup\mathbf{S}_1)\le 1, \text{\rm diam\,}(f(\mathbf{S}_0)\cup
f(\mathbf{S}_1))\le 1$;

(2)\ \ \ the derivative of $f$, $Df_{(x,y)}=\begin{pmatrix}
                                              a & b \\
                                                c & d
                                            \end{pmatrix}
                                             $,  where $a, b, c$ and $d$ are $C^1$-functions, satisfies all over
$\mathbf{S}_0\cup\mathbf{S}_1$

$\qquad$ (a) $\text{\rm det } Df=ad-bc=1$,

$\qquad$ (b) $|d|<1<|a|\le C^*/\varepsilon,$

$\qquad$ (c) $|b|, |c|\le \varepsilon (|a|-1);$

(3) the $C^1$-functions on $f(\mathbf{S}_0)\cup f(\mathbf{S}_1)$, $\tilde{a}=a\circ
f^{-1}$, $\tilde{b}=b\circ f^{-1}$, $\tilde{c}=c\circ f^{-1}$ and
$\tilde{d}=d\circ f^{-1}$, i.e. $Df^{-1}_{(x,y)}=\begin{pmatrix}
                                              \tilde{d} & -\tilde{b} \\
                                               -\tilde{c}  & \tilde{a}
                                            \end{pmatrix}
                                             $, satisfy

$\qquad$ (a) $\left| \frac{\partial \tilde{b}}{\partial
x}\right|=\left| \frac{\partial \tilde{d}}{\partial y}\right|,
\left| \frac{\partial \tilde{b}}{\partial y}\right|, \left|
\frac{\partial \tilde{c}}{\partial x}\right|, \left| \frac{\partial
\tilde{a}}{\partial x}\right|=\left| \frac{\partial
\tilde{c}}{\partial y}\right|\le \gamma (|\tilde{a}|-1),$

$\qquad$ (b) $\left| \frac{\partial {a}}{\partial y}\right|=\left|
\frac{\partial b}{\partial x}\right|, \left| \frac{\partial
{b}}{\partial y}\right|, \left| \frac{\partial{c}}{\partial
x}\right|, \left| \frac{\partial {c}}{\partial y}\right|=\left|
\frac{\partial d}{\partial x}\right|\le \gamma (|{a}|-1),$

$\qquad$ (c) $\left| \frac{\partial \tilde{a}}{\partial y}\right|,
\left| \frac{\partial \tilde{d}}{\partial x}\right|\le \gamma
|\tilde{a}|(|\tilde{a}|-1),$

$\qquad$ (d)  $\left| \frac{\partial {a}}{\partial x}\right|, \left|
\frac{\partial {d}}{\partial y}\right|\le \gamma
|{a}|(|{a}|-1);$

(4) the variation of $\log |a(x,y)|$ in each rectangle $S_i$ is less
or equal to $\gamma (1-\alpha_i^{-1})$, where
$\alpha_i=\max_{(x,y)\in S_i}|a(x,y)|$;

(5) finally, the gap sizes satisfy:
$$
\text{\rm dist}\,(\mathbf{S}_0\cup\mathbf{S}_1)\ge \frac{\varepsilon}{\gamma} \ \ \ \ \
\text{and}\ \ \ \ \ \text{\rm dist}\,(f(\mathbf{S}_0),f(\mathbf{S}_1))\ge
\frac{\varepsilon}{\gamma}.
$$
\end{Def}

 The nice feature of the maps from
$\mathcal{F}(C^*, \varepsilon, \gamma)$ is that the stable and
unstable foliations have small uniformly bounded distortion, as the
following theorem shows.
\begin{Thm}\label{distortiontheorem}
For a given $C^*>0$ and all small enough $\varepsilon>0$ and
$\gamma>0$, given $f\in \mathcal{F}(C^*, \varepsilon, \gamma),$ the
basic set $\Lambda(f)$ gives dynamically defined Cantor sets
$(K^u,\psi^u)$ and $(K^s,\psi^s)$ with distortion, bounded by
$D(C^*, \varepsilon, \gamma)=4(C^*+3)\gamma+2\varepsilon$. In
particular,
$$
e^{-D(C^*, \varepsilon, \gamma)}\tau_{L}(\mathcal{P}^s)\le
\tau_{L}(K^s(f))\le e^{D(C^*, \varepsilon,
\gamma)}\tau_{L}(\mathcal{P}^s),
$$
$$
e^{-D(C^*, \varepsilon, \gamma)}\tau_{R}(\mathcal{P}^s)\le
\tau_{R}(K^s(f))\le e^{D(C^*, \varepsilon,
\gamma)}\tau_{R}(\mathcal{P}^s),
$$
$$
e^{-D(C^*, \varepsilon, \gamma)}\tau_{L}(\mathcal{P}^u)\le
\tau_{L}(K^u(f))\le e^{D(C^*, \varepsilon,
\gamma)}\tau_{L}(\mathcal{P}^u),
$$
$$
e^{-D(C^*, \varepsilon, \gamma)}\tau_{R}(\mathcal{P}^u)\le
\tau_{R}(K^u(f))\le e^{D(C^*, \varepsilon,
\gamma)}\tau_{R}(\mathcal{P}^u).
$$
\end{Thm}
\brm Again, for $C^*=2$ this Theorem coincides with Theorem 2 in \cite{Du2}. Notice that conditions (2b) and
(2c) of definition \ref{classofmaps} imply that $|b|, |c| \le C^*$, and for $C^*=2$ this gives an unreasonable
restriction on the class of maps that could be considered. We will need to apply Theorem \ref{distortiontheorem}
for a map which belongs to the class $\mathcal{F}(C^*, \varepsilon, \gamma)$ with larger value of $C^*$, see
Proposition \ref{tbelongstoclassf}. \erm
\begin{proof}[Proof of Theorem \ref{distortiontheorem}.] The straightforward
repetition of the proof of Theorem 2 in \cite{Du2} with the necessary adjustments needed to take the constant
$C^*$ into account proves Theorem \ref{distortiontheorem}. The only place in the proof of Theorem 2 from
\cite{Du2} where the condition $|a|\le \frac{2}{\varepsilon}$ is used is the inequality (3) from Lemma 4.2. If
we use the inequality $|a|\le \frac{C^*}{\varepsilon}$ instead, $6\gamma$ should be replaced by
$\frac{3}{2}(C^*+2)$ there. Due to this change, in Lemma 4.1 from \cite{Du2} one should take $2(C^*+2)\gamma$ instead of
$8\gamma$ as an upper bound of Lipschitz seminorm $Lip(\sigma^s)$ and $Lip(\sigma^u)$  of functions $\sigma^s$
and $\sigma^u$ that describe stable and unstable foliations. This leads to similar changes in the statement of
Lemma 4.4 from \cite{Du2} and in the estimate of the distortion. Finally we use Lemma \ref{dist}  to finish the proof of Theorem
\ref{distortiontheorem}. \end{proof}

\section{Birkhoff and Gelfreich normal forms}\label{s.splitting}

In this section we collect some known results on quadratic families, Birkhoff normal form, and splitting of separatrices, in many cases reformulating them to adapt to our setting, so preparing to use them later in Section \ref{s.thickhorseshoe}.

 \subsection{Uniqueness of the area preserving quadratic family}

The non-trivial quadratic diffeomorphism of the plane which preserve area and orientation with a fixed point can be put after a linear change of coordinates into the Henon family (\ref{hf}), see \cite{H}.
  In particular,
 we can consider the family
 \beq\label{initialfamili}
F_{\varepsilon} :(x,y)\mapsto (x+y-x^2+\varepsilon,\
y-x^2+\varepsilon)
  \eneq
  instead of (\ref{hf}). In this form it is a partial case of a so
  called {\it generalized standard family}, and it was considered in
  \cite{G1}.

  Another form of the quadratic area preserving family\footnote{See Section 4 in \cite{Du4} for an explicit change of coordinates sending the family $\widetilde{F}_{\varepsilon}$ to the Henon family (\ref{hf}). Notice also that in Subsection \ref{ss.rescaling} the family (\ref{initialfamili}) is send to the family (\ref{family}) by an affine change of coordinates, and  the change of coordinates $(u,v)\mapsto (2x, 2\sqrt{2} y)$ together with reparametrization $\varepsilon=\sqrt{2}\delta$ conjugates the family (\ref{family}) with the family (\ref{initialfamili2}).}, namely
  \beq\label{initialfamili2}
\widetilde{F}_{\varepsilon} :(x,y)\mapsto (x+\varepsilon (y+\varepsilon(x-x^2)),\
y+\varepsilon(x-x^2)),
  \eneq
  was used in \cite{GSa}.

  \subsection{Rescaling and the family of maps close to identity}\label{ss.rescaling}

Let us start with the family $F_{\varepsilon}$ (\ref{initialfamili}). Consider the following family of the affine coordinate
changes:
$$
\Upsilon_{\delta}\begin{pmatrix}
               u \\
               v \\
             \end{pmatrix}=-\begin{pmatrix}
               \delta^2 \\
               0 \\
             \end{pmatrix}+\begin{pmatrix}
                             \delta^2 & 0 \\
                             0 & \delta^3 \\
                           \end{pmatrix}\begin{pmatrix}
               u \\
               v \\
             \end{pmatrix},
$$
where $\delta=\varepsilon^{\frac{1}{4}}$. Then
$$
\Upsilon_{\delta}^{-1}\circ F_{\delta^4}\circ
\Upsilon_{\delta}\begin{pmatrix}
               u \\
               v \\
             \end{pmatrix}=\begin{pmatrix}
               u+\delta v \\
               v+\delta(2u-u^2) \\
             \end{pmatrix}+\delta^2\begin{pmatrix}
               2u-u^2 \\
               0 \\
             \end{pmatrix}.
$$
 Now we have a
family of area preserving maps close to identity. For each of these maps the origin
is a saddle with eigenvalues
$$
\lambda_1=1+\delta^2+\sqrt{\delta^4+2\delta^2}=1+\sqrt{2}\delta
+O(\delta^2)>1,
$$
$$
\lambda_2=\lambda_1^{-1}=1+\delta^2-\sqrt{\delta^4+2\delta^2}=1-\sqrt{2}\delta
+O(\delta^2)<1.
$$
Set $h=\log \lambda_1$. By definition
$h=\sqrt{2}\delta+O(\delta^2)$, and $\delta$ can be given by
implicit function of $h$. Define the following (rescaled and
reparametrized)  family \beq\label{family} \frak{F}_h: (u,v)\mapsto
(u,v)+\delta(v,2u-u^2)+\delta^2(2u-u^2,0), \ \delta=\delta(h). \eneq

\subsection{Birkhoff normal form}

A real analytic area preserving
diffeomorphism of a two dimensional domain in a neighborhood of a
saddle with eigenvalues $(\lambda, \lambda^{-1})$ by an analytic
change of coordinate can be reduced to the Birkhoff normal form
(\cite{S}, see also \cite{SM}): \beq N(x,y)=(\Delta(xy)x,
\Delta^{-1}(xy)y), \eneq where
$\Delta(xy)=\lambda+a_1xy+a_2(xy)^2+\ldots$ is analytic.

We need a generalization of this Birkhoff normal form for one-parameter families. In the following statement we combine the results of Proposition 3.1 from \cite{FS} and of Proposition 4 from \cite{Du4}.

\bthm
There exists a neighborhood $U$ of the origin such that for all
$h\in (0,h_0)$ there exists a coordinate change $C_h:U\to \widehat{U}$  with
the following properties:

1. If $N_h=C_h\frak{F}_hC_h^{-1}$ then $N_h(u,v)=(\Delta_h(uv)u, \Delta_h^{-1}(uv)v)$, where
$\Delta_h(uv)=\lambda(h)+a_1(h)uv+a_2(h)(uv)^2+\ldots$ is analytic.

2. $C^3$-norms of the coordinate changes $C_h$ are uniformly bounded
with respect to the parameter $h$.

3. $\Delta_h(s)\ge 1$ is a smooth function of $s$ and $h$.
 \ethm

\brm The second property is not formulated explicitly in \cite{FS} or in \cite{Du4} but it immediately follows from Cauchy estimates.
Indeed, it follows from the proof there that the map $C_h$ is analytic and radius of convergence of the
corresponding series is uniformly bounded from below. \erm

Also we will need the following property of the parametric Birkhoff normal form for the family $\frak{F}_h$.

\blm\label{analogduarte} For some constant $C>0$ and  small enough
$h_0>0$ and $s_0>0$ the following holds. For all $h\in [0, h_0)$ and
$s\in [0,s_0)$

1. $\log \Delta_h(s)\ge C^{-1}h,$

2. $|\Delta_h'(s)|\le Ch,$

3. $|\Delta_h''(s)|\le Ch.$\elm

\brm This Lemma is similar to Lemma 6.3 from \cite{Du2}, but in our case we have one, not two parameter family,
and therefore those two statements are essentially different.  \erm
\begin{proof}[Proof of Lemma \ref{analogduarte}.]
Consider $g(s,h)=\log \Delta_h(s)$. We have $g(s,0)=0, $ $g(0,h)=h$, and $g$ is a smooth function of $(s,h)$.
This implies that for small enough $s_0>0$, $h_0>0$ and large $C>0$ we have $g(s,h)\ge C^{-1}h$ for all $s\in
[0,s_0]$ and $h\in [0,h_0]$.

 From
the explicit form of the family $\frak{F}_h$ (\ref{family}) we see
that $\frak{F}_h\to \text{\rm Id}$ as $h\to 0$ in $C^r$-norm for
every $r\in \Bbb{N}$. Since $C^3$-norms of $C_h$ and $C_h^{-1}$ are
uniformly bounded, this implies that $N_h\to \text{\rm Id}$ in
$C^2$-norm as $h\to 0$. In particular,
$$
DN_h(x,y)=\begin{pmatrix}
            \Delta_h(xy)+\Delta'_h(xy)xy & x^2\Delta_h'(xy) \\
            -\frac{\Delta_h'(xy)y^2}{\Delta^2_h(xy)} & \Delta_h^{-1}(xy)-\frac{\Delta_h'(xy)xy}{\Delta^2(xy)} \\
          \end{pmatrix}\to \begin{pmatrix}
                             1 & 0 \\
                             0 & 1 \\
                           \end{pmatrix}
$$
as $h\to 0$ uniformly in $(x,y)\in U$.  This implies that $\Delta_h'(s)\to 0$ as $h\to 0$ uniformly in $s\in
[s_1, s_0]$ for every $s_1\in (0,s_0)$. Also $\Delta_0(s)=1$ for every $h\in [0, h_0)$, so $\Delta_0'(s)=0$.
Since $\Delta_h'(s)$ is a continuous function, this implies that $\Delta_h'(s)\to 0$ as $h\to 0$ uniformly in
$s\in [0, s_0]$. Since $\Delta_h'(s)$ is a smooth function of $(s,h)$, this implies that $|\Delta_h'(s)|\le Ch$
if $C>0$ is large enough. Similarly one can show that $|\Delta_h''(s)|\le Ch$. \end{proof} 

\begin{figure}
  \includegraphics[width=1\textwidth]{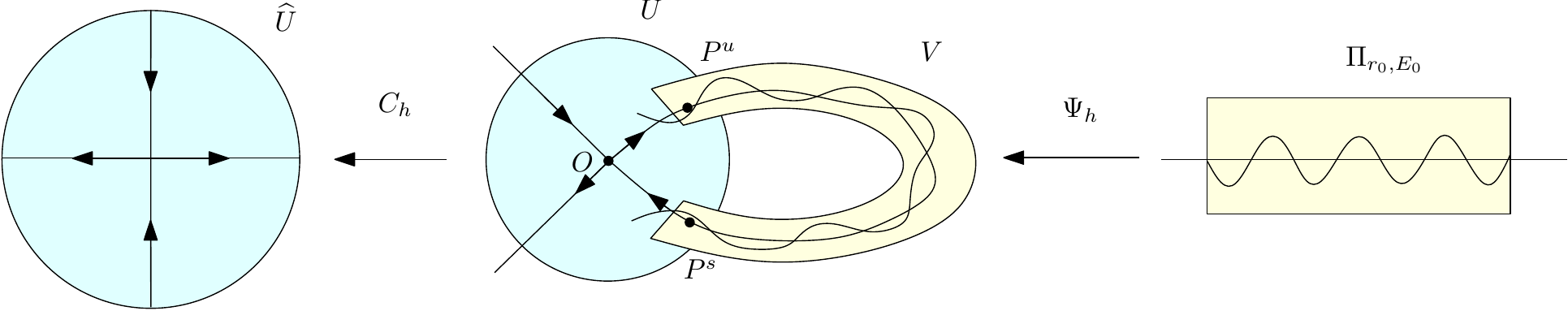}\\
  \caption{Birkhoff and Gelfreich normal forms for $\frak{F}_h$}\label{fig.normalforms}
\end{figure}

\subsection{Gelfreich normal form and splitting of separatrices}

The family $\frak{F}_h$ is closely related to the conservative vector field \beq\left\{
                \begin{array}{ll}\label{vectorfield}
                  \dot x=y,  \\
                  \dot y=2x-x^2.
                \end{array}
              \right.
\eneq

 Namely, due to Theorems ${\rm A}$ and ${\rm A'}$ from \cite{FS1} (see also
Proposition 5.1 from \cite{FS}) the separatrix phase curve of the vector field (\ref{vectorfield}) (let us
denote it by $\sigma$) gives a good approximation of some finite pieces of $W_{{\frak{F}_h}}^s(0,0)$ and $W_{\frak{F}_h}^u(0,0)$. Denote by
$\widetilde{\sigma}$ a segment of separatrix $\sigma$ that contains some points $P^u\in W^u_{loc}(0,0)\cap U$
and $P^s\in W^s_{loc}(0,0)\cap U$  and by $V$ a neighborhood of $\widetilde{\sigma}$. Denote by
$\widetilde{W}_h^s(0,0)$ the finite piece of $W_{\frak{F}_h}^s(0,0)$ between the points where $W^s_{\frak{F}_h}(0,0)$ leaves $U$ for
the first time and the first point where $W^s_{\frak{F}_h}(0,0)$ returns to $U$ again. Define $\widetilde{W}^u_h(0,0)$ in a
similar way. Then $\widetilde{W}_h^s(0,0)$ and $\widetilde{W}_h^u(0,0)$ are always (for all $h\in (0,h_0)$) in
$V$.

The restriction of the map ${\frak{F}_h}$ on the local separatrix
$W_{loc}^u(0,0)$ is conjugated with a multiplication
$\xi\mapsto \lambda\xi$, $\xi\in (\Bbb{R}, 0)$. Let us call a
parameter $t$ on $W_{{\frak{F}_h}}^u(0,0)$ {\it standard} if it is
obtained by a substitution of $e^t$ instead of $\xi$ into the
conjugating function. Such a parametrization is defined up to a
substitution $t\mapsto t+const$.

Denote $\Pi_{r_0, E_0}=\{(t,E)\in \Bbb{R}^2\ |\ |\,t|<r_0,
|\,E|<E_0\}$.

In the following Theorem we summarize the results from \cite{G1, G2, G3, GSa, BG}.
 \bthm\label{t.Gelfreich}  There
are neighborhood $V$ of the segment of $\sigma$ between points $P^u$
and $P^s$  and constants $r_0$ and $E_0$ such that for some $h_0>0$
and all $h\in (0,h_0)$ there exists a map $\Psi_h:\Pi_{r_0, E_0}\to
\Bbb{R}^2$ with the following properties:

1. $\Psi_h(\Pi_{r_0, E_0})\supset V$;

2. $\Psi_h$ is real analytic;

3. $\Psi_h$ is area preserving;

4. $\Psi_h$ conjugates the map $\frak{F}_h$ with the shift $H_h:(t,E)\mapsto (t+h, E)$;

5. $C^3$-norms of $\Psi_h$ and $\Psi_h^{-1}$ are uniformly bounded with respect to $h\in (0,h_0)$.

6. $\Psi^{-1}_h(\widetilde{W}_h^u)=\{E=0\}$, and $t$ gives a
standard parametrization of the unstable manifold;

7. Stable manifold
$\Psi_h^{-1}(\widetilde{W}^s_h)$ can be represented as a graph of a
real-analytic $h$-periodic function $\Theta(t)$ such that
\begin{align}\label{thetastar}
    \Theta(t)=8\sqrt{2}|\,\Theta_1|\,h^{-6}e^{-2\pi^2\slash h}\sin \frac{2\pi
t}{h}+O(h^{-5}e^{-2\pi^2\slash h});\\
\dot\Theta(t)=16\sqrt{2}\pi|\,\Theta_1|\,h^{-7}e^{-2\pi^2\slash h}\cos \frac{2\pi
t}{h}+O(h^{-6}e^{-2\pi^2\slash h});\\
\ddot\Theta(t)=-32\sqrt{2}\pi^2|\Theta_1|\,h^{-8}e^{-2\pi^2\slash h}\sin \frac{2\pi
t}{h}+O(h^{-7}e^{-2\pi^2\slash h});
\end{align}

8. $|\,\Theta_1|\ne 0$.
\ethm

In \cite{G3} existence of the normal form $\Psi_h$ that satisfies properties 1.-6. was shown. In \cite{G1, G2}  the splitting of separatrices (the form of the function $\Theta(t)$) was studied for the initial family $F_{\varepsilon}$ (\ref{initialfamili}). The fact that $|\,\Theta_1|\ne 0$ was proved in \cite{GSa}.
 In the recent paper \cite{BG} the whole asymptotic series for $\Theta(t)$ is presented (in fact, for a much wider class of families that includes area preserving Henon  family), but here we stated only the partial case of that result which we will need in Section \ref{s.thickhorseshoe}.

\brm\label{remarkaboutmu} To simplify the notation define the function \beq \mu(h)=16\sqrt{2}\pi
|\,\Theta_1|\,h^{-7}\exp{(-2\pi^2\slash h)}. \eneq Notice that the angle between $\widetilde{W}_h^u$ and
$\widetilde{W}_h^s$ at the homoclinic point in the normalized coordinates is equal to $\mu(h)(1+O(h))$. The function $\Theta(t)$ can now be represented
in the following way: $$ \Theta(t)=\frac{1}{2\pi} h \mu(h)\sin
\frac{2\pi t}{h}+O(h^2\mu(h)), \ \   \dot\Theta(t)=\mu(h)\cos \frac{2\pi t}{h}+O(h\mu(h)), \ \ $$ $$
\ddot\Theta(t)=-2\pi h^{-1} \mu(h)\sin \frac{2\pi t}{h}+O(\mu(h)).
$$\erm

\section{Construction of a thick horseshoe for area preserving Henon family}\label{s.thickhorseshoe}

Here we prove Theorem \ref{t.henonfamily}. In order to do so we construct the horseshoe for the first return map in a neighborhood of a saddle, verify the conditions of the Duarte's Distortion Theorem from Section \ref{s.distest}, and obtain estimates on one-sided thicknesses of the constructed horseshoe. Relations between one-sided thicknesses and Hausdorff dimension obtained in Section \ref{s.leftrightCantor} will imply the required result.

\subsection{Construction of the domain for the first return map}

Let $q_h^u$ be the closest to $P^u\in \sigma$ point of intersection of $\widetilde{W}_h^u$ and
$\widetilde{W}_h^s$. Consider a finite sequence of images of $q_h^u$ under the map $\frak{F}_h$ that belong to
the neighborhood $V$, $\{q_h^u, \frak{F}(q_h^u), \frak{F}^2(q_h^u), \ldots\}$. Let $q_h^s$ be the point of this
sequence closest to the point $P^s\in \sigma$. Define $k(h)\in \Bbb{N}$ by $\frak{F}_h^{k(h)}(q_h^u)=q^s_h$.
Take the vector $v=(1,0)\in T_{C_h(q^u_h)}\widehat{U}$ and consider the vector $w=(w_1, w_2)=D(C_h\circ \Psi_h\circ
H^{k(h)}\circ \Psi_h^{-1}\circ C_h^{-1})v\in T_{C_h(q_h^s)}\widehat{U}$. Without loss of generality we can assume that $w_1>0$
(otherwise just take a homoclinic point between $q_h^u$ and $\frak{F}(q_h^u)$ instead of $q_h^u$). Scaling, if
necessary, we can assume that in the Birkhoff normalizing coordinates we have $C_h(q_h^u)=(1,0)$, $C_h(q_h^s)=(0,1)$.

Fix small $\nu>0$. Recall that $\lambda=\Delta_h(0)=e^h$. Set \beq n=\left[-\frac{\log
(\mu(h)h^{1+\nu})}{2h}\right]. \eneq Due to this choice $\lambda^{-2n}\approx \mu(h)h^{1+\nu}$. More precisely,
$\lambda^{-2n}\in [\mu(h)h^{1+\nu}, \lambda^2\mu(h)h^{1+\nu})$.

\brm Notice that this choice of $n$ for $\nu=\frac{1}{2}$ is analogous to the formula (7) in \cite{Du2}.  \erm

Define the following lines:
\begin{align*}
    \tau_{(1,0)}^+=\{x=\lambda^{\frac{1}{10}}\},\ \ \ \
\tau_{(1,0)}^-=\{x=\lambda^{-\frac{1}{10}}\},\\
\tau_{(0,1)}^+=\{y=\lambda^{\frac{1}{10}}\},\ \ \ \
\tau_{(0,1)}^-=\{y=\lambda^{-\frac{1}{10}}\}.
\end{align*}
Denote by $S$ the square formed by coordinate axes, 
$N_h^n(\tau^+_{(0,1)}),$
and
$N_h^{-n}(\tau^+_{(1,0)})$. The bottom and left edges of $S$ have the size  \beq\label{e.defl} l=\lambda^{-n+\frac{1}{10}}.
\eneq Notice that since $DN_h$ is close to the linear map $\begin{pmatrix}
                                                             \lambda & 0 \\
                                                             0 & \lambda^{-1} \\
                                                           \end{pmatrix},$ the curve $N_h^n(\tau^+_{(0,1)})$ (resp.,
$N_h^{-n}(\tau^+_{(1,0)})$) is $C^1$-close to a horizontal (resp., vertical) line.  Denote by $R^u$ and $R^s$
the rectangles formed by $x$-axis, 
$\tau_{(1,0)}^+$, $\tau_{(1,0)}^-$ and $N_h^{2n}(\tau^+_{(0,1)})$, and by $y$-axis, 
$\tau_{(0,1)}^+$, $\tau_{(0,1)}^-$ and $N_h^{-2n}(\tau^+_{(1,0)})$, respectively. Notice that
$R^u=N_h^{n}(S)\cap \{x\ge \lambda^{-\frac{1}{10}}\}$ and $R^s=N_h^{-n}(S)\cap \{y\ge
\lambda^{-\frac{1}{10}}\}$.

\begin{figure}
  \includegraphics[width=1\textwidth]{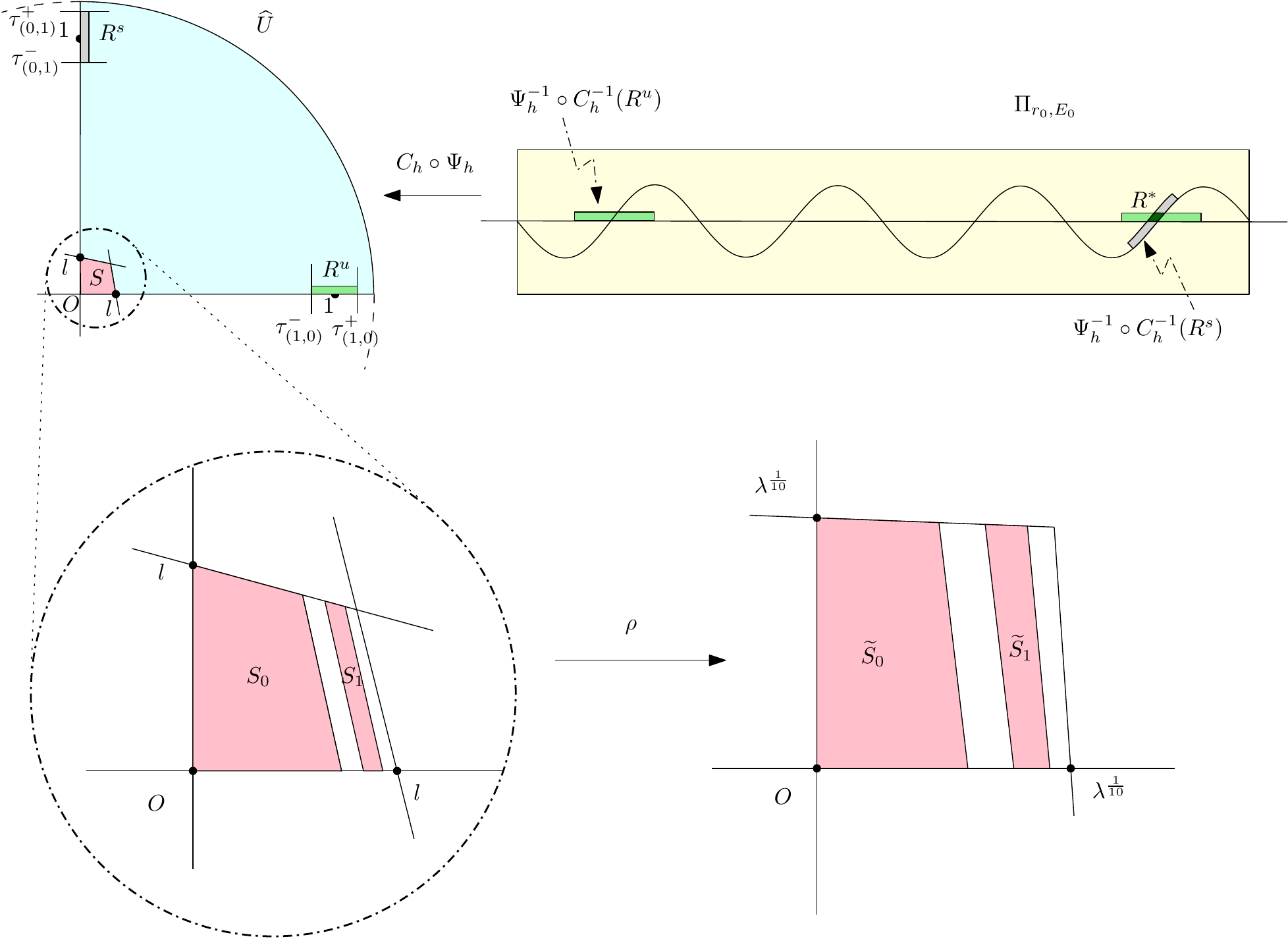}\\
  \caption{Construction of the horseshoe.}\label{fig.horseshoe}
\end{figure}

Denote by $R^*$ the intersection (see Fig. \ref{fig.horseshoe}):
$$
R^*=H^{k(h)}\circ \Psi_h^{-1}\circ C_h^{-1} (R^u)\cap
\Psi_h^{-1}\circ C^{-1}_h(R^s).
$$
Now consider the rectangles \beq S_0=S\cap N_h^{-1}(S)\ \ \text{and}\ \ \ S_1=N_h^{-n}\circ  C_h\circ \Psi_h\circ
H^{-k(h)}(R^*) \eneq and define the first return map
$$
T(x,y)=\left\{
           \begin{array}{ll}
            N_h(x,y), & \hbox{if $(x,y)\in S_0$;} \\
            N_h^{n}\circ C_h\circ \Psi_h\circ H^{k(h)}\circ  \Psi_h^{-1}\circ C_h^{-1}\circ N_h^{n}(x,y), & \hbox{if $(x,y)\in S_1$.}
          \end{array}
        \right.
$$

\subsection{Renormalization}

We are going to prove that the map $T$ has a hyperbolic invariant set in $S$ and to estimate its Hausdorff dimension with respect to the parameter $h$. It is convenient to renormalize the map $T$. Namely, using the approach from \cite{Du2}, set
$$
\rho:S\to [0,2]\times [0,2], \ \rho(x,y)=(\Delta_h^n(xy)x, \Delta_h^n(xy)y),
$$
and define
$$
\rho(S_0)=\widetilde{S}_0,\ \rho(S_1)=\widetilde{S}_1,\ \widetilde{T}:\widetilde{S}_0\cup \widetilde{S}_1\to [0,2]\times [0,2], \  \widetilde{T}=\rho\circ T\circ \rho^{-1}.
$$
Notice that $\rho^{-1}$ is defined by
$$
\rho^{-1}(x,y)=(\Delta_h^{-n}(t(xy))x, \Delta_h^{-n}(t(xy))y),
$$
where $t(s)$ is such that $t(\Delta_h^{2n}(xy)xy)=xy$, or, equivalently,
\begin{equation}\label{e.t}
t(s)\cdot \Delta_h^{2n}(t(s))=s.
\end{equation}
\blm\label{l.t}
For some $C>0$ independent of $h$ the following holds. If $(x_0,y_0)\in S$, $\rho(x_0,y_0)=(x, y)$, and $s=xy$ then
 $|t(s)|$, $|t'(s)|$, and $|t''(s)|$ are bounded by $C\lambda^{-2n}$.
\elm
\begin{proof}[Proof of Lemma \ref{l.t}] Since $(x_0, y_0)\in S$, we have
$$
0\le t(s)=t(\Delta_h^n(x_0y_0)x_0\cdot\Delta_h^n(x_0y_0)y_0)=x_0y_0\le 4\lambda^{-2n}.
$$
Differentiating (\ref{e.t}) we get
$$
t'(s)\Delta_h^{2n}(t(s))+2nt(s)\Delta_h^{2n-1}(t(s))\Delta'_h(t(s))t'(s)=1,
$$
therefore
$$
|t'(s)|=\left|\Delta_h^{-2n}(t(s))\frac{1}{1+2nt(s)\frac{\Delta'_h(t(s))}{\Delta_h(t(s))}}\right|\le 4\lambda^{-2n}
$$
Differentiating once again, we get
$$
t''(s)(\lambda^{2n}+O(n))+t'(s)O(n)=0, \ \text{\rm so}\ \ |t''(s)|\le \frac{|t'(s)|O(n)}{\lambda^{2n}|1+O(n\lambda^{-2n})|}<\lambda^{-2n}
$$
\end{proof}
Notice that
$$
\widetilde{T}|_{\widetilde{S}_0}(x,y)=\rho\circ N_h\circ \rho^{-1}(x,y)=(\Delta_h(t(xy))x, \Delta^{-1}_h(t(xy))y).
$$
\blm\label{l.Dt}
We have
$$
D\widetilde{T}|_{\widetilde{S}_0}(x,y)=\begin{pmatrix}
                                         \lambda & 0 \\
                                         0 & \lambda^{-1} \\
                                       \end{pmatrix}+\begin{pmatrix}
                                                       O(\lambda^{-2n}h) & O(\lambda^{-2n}h) \\
                                                       O(\lambda^{-2n}h) & O(\lambda^{-2n}h) \\
                                                     \end{pmatrix}.
$$
\elm
\begin{proof}[Proof of Lemma \ref{l.Dt}]
Differentiating explicitly we get
$$
D\widetilde{T}|_{\widetilde{S}_0}(x,y)=\begin{pmatrix}
                                         \Delta_h(t(xy))+xy\Delta_h'(t(xy))t'(xy) & x^2\Delta_h'(t(xy))t'(xy) \\
                                         -\frac{y^2}{\Delta_h^2(t(xy))}\Delta_h'(t(xy))t'(xy) & \Delta_h^{-1}(t(xy))-xyt'(xy)\frac{\Delta_h'(t(xy))}{\Delta_h^2(t(xy))} \\
                                       \end{pmatrix}.
$$
Now the required estimates follow from Lemmas \ref{analogduarte} and \ref{l.t}.
\end{proof}

In order to study $\widetilde{T}|_{\widetilde{S}_1}$ one can introduce the following maps:
$$
G:\widetilde{S}_1\to R^u, \ G=N_h^n\circ \rho^{-1}, \ G(x,y)=(x, \Delta_h^{-2n}(t(xy))y),
$$
and
$$
\widehat{G}: R^s \to\widetilde{S}_1, \ \widehat{G}=\rho \circ N_h^n, \ \widehat{G}(x,y)=(\Delta_h^{2n}(xy)x, y).
$$
With this notation we have
$$
\widetilde{T}|_{\widetilde{S}_1}=\widehat{G}\circ C_h\circ \Psi_h\circ H^{k(h)}\circ  \Psi_h^{-1}\circ C_h^{-1}\circ G
$$

\subsection{Cone condition}

The coordinate changes $C_h\circ \Psi_h$ and $\Psi_h^{-1}\circ
C_h^{-1}$ have uniformly bounded $C^3$-norms. Assume that their
$C^3$-norms are bounded by some constant $C_0$.

Let us introduce the following cone fields in
$\tilde{S}_0\cup\tilde{S}_1$: \beq K^u(x,y)=\{\bar{v}=(v_1, v_2)\in
T_{(x,y)}\tilde{S}_i\ |\ |\,v_1|>0.01C_0^{-6}h^{-1-\nu} |v_2|\}, \
\ \text{and} \eneq \beq K^s(x,y)=\{\bar{v}=(v_1, v_2)\in
T_{(x,y)}\tilde{S}_i\ |\ |\,v_2|>0.01C_0^{-6}h^{-1-\nu}
|v_1|\}.\eneq

\blm\label{conecondition1}{\bf (Cone condition for $\tilde{S}_0$)}
For small enough $h$ the following holds.

For every vector $\bar{v}\in K^u(x,y)$, $(x,y)\in \tilde{S}_0$, we
have $D\tilde{T}_{(x,y)}(\bar{v})\in K^u(\tilde{T}(x,y))$, and if
$D\tilde{T}_{(x,y)}(\bar{v})=\bar{w}\equiv (w_1,w_2)$ then $|w_1|\ge
\lambda^{0.9}|v_1|$.

For every vector $\bar{v}\in K^s(x,y)$, $(x,y)\in
\tilde{T}(\tilde{S}_0)$, we have
$D\tilde{T}_{(x,y)}^{-1}(\bar{v})\in K^s(\tilde{T}^{-1}(x,y))$, and
if $D\tilde{T}^{-1}_{(x,y)}(\bar{v})=\bar{w}\equiv (w_1,w_2)$ then
$|w_2|\ge \lambda^{0.9}|v_2|$. \elm

\begin{proof}[Proof of Lemma \ref{conecondition1}.] This follows directly
from Lemma \ref{l.Dt}. \end{proof}


\blm\label{conecondition2}{\bf (Cone condition for $\tilde{S}_1$)}
For small enough $h$ the following holds.

For every vector $\bar{v}\in K^u(x,y)$, $(x,y)\in \tilde{S}_1$, we
have $D\tilde{T}_{(x,y)}(\bar{v})\in K^u(\tilde{T}(x,y))$, and if
$D\tilde{T}_{(x,y)}(\bar{v})=\bar{w}\equiv (w_1,w_2)$ then
$|w_1|\ge0.01C_0^{-4}h^{-1-\nu}|v_1|$ and $|\bar{w}|\le
25C_0^4h^{-1-\nu}|\bar{v}|$.

 For every vector
$\bar{v}\in K^s(x,y)$, $(x,y)\in \tilde{T}(\tilde{S}_1)$, we have
$D\tilde{T}_{(x,y)}^{-1}(\bar{v})\in K^s(\tilde{T}^{-1}(x,y))$, and
if $D\tilde{T}^{-1}_{(x,y)}(\bar{v})=\bar{w}\equiv (w_1,w_2)$ then
$|w_2|\ge0.01C_0^{-4}h^{-1-\nu}|v_2|$ and $|\bar{w}|\le
25C_0^4h^{-1-\nu}|\bar{v}|$.
 \elm

Before to begin the proof of Lemma \ref{conecondition2} we will formulate and proof two extra  lemmas that give
estimates of the angle between images of vectors under linear maps.

\blm\label{estimateofangle} For any two vectors $\bar{u}_1, \bar{u}_2$  and any linear map $A:\mathbb{R}^2\to
\mathbb{R}^2$ the following inequality holds:
$$
\sin \angle (A\bar{u}_1, A\bar{u}_2)\le \|A\|\cdot \|A^{-1}\|\cdot|\sin \angle (\bar{u}_1, \bar{u}_2)|
$$
 \elm

\begin{proof}[Proof of Lemma \ref{estimateofangle}.] Take two vectors $\bar{s}_1$ and $\bar{s}_2$ such that
$\bar{s}_2\perp (\bar{s}_1-\bar{s}_2)$ and $\bar{s}_1\| \bar{u}_1$, $\bar{s}_2\| \bar{u}_2$. In this case $|\sin
\angle (\bar{u}_1, \bar{u}_2)|=\frac{|\bar{s}_1-\bar{s}_2|}{|\bar{s}_1|}$. Now  we have
$$
\sin \angle (A\bar{u}_1, A\bar{u}_2)\le \frac{|A\bar{s}_1-A\bar{s}_2|}{|A\bar{s}_1|}\le
\frac{\|A\||\bar{s}_1-\bar{s}_2|}{\|A^{-1}\|^{-1}|\bar{s}_1|}=\|A\|\cdot \|A^{-1}\|\cdot |\sin \angle
(\bar{u}_1, \bar{u}_2)|
$$
\end{proof} 

\blm\label{estimateofangleadd} For any vector $\bar{u}\in \Bbb{R}^2$, $\bar{u}\ne 0$, and any linear maps $A,B:
\Bbb{R}^2\to \Bbb{R}^2$ the following inequality holds:
$$
\sin \angle (A\bar{u}, B\bar{u})\le \|A\|\cdot \|A-B\|.
$$
 \elm

\begin{proof}[Proof of Lemma \ref{estimateofangleadd}.]
 $$
\sin \angle (A\bar{u}, B\bar{u})\le \frac{|A\bar{u}-B\bar{u}|}{|A\bar{u}|}\le
\frac{\|A-B\|}{\|A\|^{-1}}=\|A\|\cdot \|A-B\|.
 $$
\end{proof}

\begin{proof}[Proof of Lemma \ref{conecondition2}.] We will prove the first
part of the statement. The proof of the second part is completely
the same.

Take a vector $\bar{v}\equiv(v_1, v_2)\in K^u(x,y), \ (x,y)\in
\tilde{S}_1. $

Consider the following points:
$$
P_1=(x,y)\in \widetilde{S}_0, \ \ \ P_2=G(P_1)\in R^u, \ \ \
P_3=\Psi_h^{-1}\circ C_h^{-1}(P_2)\in \Pi_{r_0, E_0},
$$
$$
P_4=H^{k(h)}(P_3)\in \Pi_{r_0, E_0},\ \ \ P_5=C_h\circ
\Psi_h(P_4)\in R^s, \ \ \ P_6=\widehat{G}(P_5)\in T(\widetilde{S}_1),
$$
and denote by $(x_i, y_i)$ the coordinates of the point $P_i$, $i=1,
\ldots, 6$.
 We will follow the image of the vector along this
sequence of points and estimate the angle between that image and
coordinate axes and the size of the image. Denote by $\bar{v}^{(i)}=({v}^{(i)}_1, {v}^{(i)}_2)$ the image of $\bar{v}$ at the point $P_i$, $i=1, \ldots, 6$. In particular, $\bar{v}^{(1)}=\bar{v}$.

{\it Step 1.} Let us first estimate the norm and inclination of the vector $\bar{v}^{(2)}=DG(\bar{v}^{(1)})$. We have
$$
DG(x,y)=\begin{pmatrix}
  1 & 0 \\
  \frac{2ny^2\Delta_h'(t(xy))t'(xy)}{\Delta_h^{2n+1}(t(xy))} & \Delta_h^{-2n}(t)-\frac{2nxy\Delta'_h(t(xy))t'(xy)}{\Delta_h^{2n+1}(t(xy))} \\
\end{pmatrix}=\begin{pmatrix}
                1 & 0 \\
                O(n\lambda^{-4n}h) & \lambda^{-2n}+O(n\lambda^{-4n}h) \\
              \end{pmatrix}.
$$
Therefore
$$
\bar{v}^{(2)}=DG(\bar{v}^{(1)})=\begin{pmatrix}
                                  v_1 \\
                                  v_1O(n\lambda^{-4n}h)+v_2(\lambda^{-2n}+O(n\lambda^{-4n}h)) \\
                                \end{pmatrix},
$$
and hence (since $\bar{v}\equiv(v_1, v_2)\in K^u(x,y)$) we have $\frac{1}{2}|\bar{v}^{(1)}|\le |\bar{v}^{(2)}|\le 2|\bar{v}^{(1)}|$, and
$$
\frac{|{v}^{(2)}_2|}{|{v}^{(2)}_1|}=O(n\lambda^{-4n}h)+\lambda^{-2n}\frac{v_2^{(1)}}{v_1^{(1)}}
+\frac{v_2^{(1)}}{v_1^{(1)}}O(n\lambda^{-4n}h)<200\lambda^{-2n}C_0^6h^{1+\nu}
$$

{\it Step 2.} We have $\bar{v}^{(3)}=D_{P_2}(\Psi^{-1}_h\circ C_h^{-1})\bar{v}^{(2)}$. Therefore
$C_0^{-1}|\bar{v}^{(2)}|\le |\bar{v}^{(3)}|\le C_0|\bar{v}^{(2)}|$. Let us estimate the angle between
$\bar{v}^{(3)}$ and the vector $\bar{e}_1=(1,0)$. Let $P^*$ be a projection of the point $P_2$ to the line
$\{y=0\}$. Then $\ \text{dist}(P_2, P^*)\le 2\lambda^{-2n}$. Since the image of the
line $\{y=0\}$ under the map $\Psi^{-1}_h\circ C_h^{-1}$ is a line $\{E=0\}$, the image of the vector
$\bar{e}_1=(1,0)$ under the differential $D(\Psi^{-1}_h\circ C_h^{-1})$ has the form $(s,0)=s\bar{e}_1$. Now we
have
$$
\angle (\bar{v}^{(3)}, \bar{e}_1)=\angle (\bar{v}^{(3)}, s\bar{e}_1)=\angle\left(D_{P_2}(\Psi^{-1}_h\circ
C_h^{-1})\bar{v}^{(2)}, D_{P^*}(\Psi^{-1}_h\circ C_h^{-1})\bar{e}_1\right)\le
$$
$$
\le \angle \left(D_{P_2}(\Psi^{-1}_h\circ C_h^{-1})\bar{v}^{(2)}, D_{P_2}(\Psi^{-1}_h\circ
C_h^{-1})\bar{e}_{1}\right)+\angle\left(D_{P_2}(\Psi^{-1}_h\circ C_h^{-1})\bar{e}_{1}, D_{P^*}(\Psi^{-1}_h\circ
C_h^{-1})\bar{e}_{1}\right)
$$
Now let us estimate each of the summands. Since all the angles that we consider are small, we can always assume
that $\alpha<2\sin\alpha<2\alpha$ for all angles $\alpha$ that we consider. Due to Lemma \ref{estimateofangle}
we have
\begin{multline}
\angle \left(D_{P_2}(\Psi^{-1}_h\circ C_h^{-1})\bar{v}^{(2)}, D_{P_2}(\Psi^{-1}_h\circ
C_h^{-1})\bar{e}_{1}\right)\le
\\
\le 2\sin \angle \left(D_{P_2}(\Psi^{-1}_h\circ C_h^{-1})\bar{v}^{(2)}, D_{P_2}(\Psi^{-1}_h\circ
C_h^{-1})\bar{e}_{1}\right)\le
\\
\le 2C_0^2|\sin \angle (\bar{v}^{(2)}, \bar{e}_1)|\le 2C_0^2\cdot 200\lambda^{-2n}C_0^6h^{1+\nu} =400C_0^8\lambda^{-2n}h^{1+\nu}
\end{multline}
Due to Lemma \ref{estimateofangleadd} we have
\begin{multline}
    \angle\left(D_{P_2}(\Psi^{-1}_h\circ C_h^{-1})\bar{e}_{1}, D_{P^*}(\Psi^{-1}_h\circ
C_h^{-1})\bar{e}_{1}\right)\le \\
\le 2\sin \angle\left(D_{P_2}(\Psi^{-1}_h\circ C_h^{-1})\bar{e}_{1}, D_{P^*}(\Psi^{-1}_h\circ
C_h^{-1})\bar{e}_{1}\right)\le\\
\le 2C_0\cdot C_0\,\text{dist}(P_2, P^*)\le 4C_0^2
\lambda^{-2n}
\end{multline}
Finally (if $h$ is small enough and $\lambda=e^h$ is close to 1) we
have \beq\label{angleestimatevthree} \angle (\bar{v}^{(3)},
\bar{e}_1)\le 400C_0^8\lambda^{-2n}h^{1+\nu}+4C_0^2\lambda^{-2n}< 5C_0^2\lambda^{-2n}
\eneq

 {\it Step 3.} Since $H(t, E)=(t+h, E)$, the estimates for $\bar{v}^{(3)}$ work for
  $\bar{v}^{(4)}=DH^{k(h)}(\bar{v}^{(3)})$ also.

{\it Step 4.} Consider $\bar{v}^{(5)}=D_{P_4}(C_h\circ
\Psi_h)\bar{v}^{(4)}\in T_{P_5}\widehat{U}$. Notice that \beq
|\bar{v}^{(5)}|\ge C_0^{-1}|\bar{v}^{(4)}|\ge
C_0^{-2}|\bar{v}^{(2)}|>\frac{1}{2}\,C_0^{-2} |\bar{v}^{(1)}|\ \
\ \ \text{and} \eneq \beq |\bar{v}^{(5)}|\le C_0|\bar{v}^{(4)}|\le
C_0^2|\bar{v}^{(2)}|\le 2C_0^2|\bar{v}^{(1)}|.  \eneq Now let us
estimate the angle between $\bar{v}^{(5)}$ and the axis $Oy$. Let
$P^{\#}$ be a projection of the point $P_5$ on the line $\{x=0\}$.
Take the vector $\bar{e}_2=(0,1)\in T_{P^{\#}}\widehat{U}$ and consider the
image $D_{P^{\#}}(\Psi_h^{-1}\circ C_h^{-1})\bar{e}_2\in
T_{\Psi_h^{-1}\circ C_h^{-1}(P^{\#})}\Pi_{r_0, E_0}$. The vector
$D_{P^{\#}}(\Psi_h^{-1}\circ C_h^{-1})\bar{e}_2$ is tangent to the
graph of the function $\Theta(t)$, and due to (\ref{thetastar}) \beq
\frac{1}{2} \mu(h)< \angle (D_{P^{\#}}(\Psi_h^{-1}\circ
C_h^{-1})\bar{e}_2, \bar{e}_1)<2\mu(h).  \eneq From
(\ref{angleestimatevthree}) we have
$$
\frac{1}{5}\mu(h)< \angle (D_{P^{\#}}(\Psi_h^{-1}\circ C_h^{-1})\bar{e}_2, \bar{v}^{(4)})<5\mu(h).
$$
Notice that $\ \text{dist}(\Psi_h^{-1}\circ C_h^{-1}(P^{\#}),
P_4)\le 2C_0\lambda^{-2n}\le 4C_0\mu(h)h^{1+\nu}$.
This implies (in the way similar to Step 3) that for small enough
$h$ \beq\label{angle5e2} \angle (\bar{v}^{(5)}, \bar{e}_2)<
5C_0^2\mu(h)+C_0\cdot C_0\cdot 4C_0\mu(h)h^{1+\nu}<6C_0^2\mu(h),
 \eneq
\beq\label{angle5e2next} \angle (\bar{v}^{(5)},
\bar{e}_2)>\frac{1}{5}\mu(h)C_0^{-2}-4C_0^3\mu(h)h^{1+\nu}>\frac{1}{6}C_0^{-2}\mu(h).
\eneq

 {\it Step 5.} Since $\widehat{G}(x,y)=(\Delta_h^{2n}(xy)x, y)$, we have
\begin{multline}
 D\widehat{G}(x,y)=\begin{pmatrix}
           \Delta_h^{2n}(xy)+2nxy\Delta_h^{2n-1}(xy)\Delta_h'(xy) & 2nx^2\Delta_h^{2n-1}(xy)\Delta_h'(xy) \\
           0 & 1 \\
         \end{pmatrix}=\\
         =\begin{pmatrix}
                         \lambda^{2n}(1+O(n\lambda^{-2n}h)) & O(n\lambda^{-2n}h)) \\
                         0 & 1 \\
                       \end{pmatrix},
\end{multline}
 hence
 $$
 D\widehat{G}(\bar{v}^{(5)})=\begin{pmatrix}
                               v_1^{(6)} \\
                               v_2^{(6)} \\
                             \end{pmatrix}=\begin{pmatrix}
                                             \lambda^{2n}v_1^{(5)}(1+O(n\lambda^{-2n}h))+v_2^{(5)}O(n\lambda^{-2n}h)) \\
                                            v_2^{(5)} \\
                                           \end{pmatrix}.
 $$
 Therefore  $|v_2^{(6)}|=|\bar{v}^{(5)}|\le 2C_0^2|\bar{v}^{(1)}|$ and
\begin{multline}
 |v_1^{(6)}|\ge |v_2^{(5)}|\cdot \left|\lambda^{2n}\frac{v_1^{(5)}}{v_2^{(5)}}(1+O(n\lambda^{-2n}h))+O(n\lambda^{-2n}h)\right|\ge \\ \ge \frac{1}{2}|\bar{v}^{(5)}|\left|\frac{1}{6}C_0^{-2}h^{-1-\nu}(1+O(n\lambda^{-2n}h))+O(n\lambda^{-2n}h)\right|\ge \frac{1}{20}C_0^{-2}h^{-1-\nu}|\bar{v}^{(5)}| \ge \frac{1}{40}C_0^{-4}h^{-1-\nu}|\bar{v}^{(1)}|
\end{multline}
This implies that $\frac{ |v_1^{(6)}|}{ |v_2^{(6)}|}>\frac{1}{80}C_0^{-6}h^{-1-\nu}$, and hence $\bar{v}^{(6)}\in K^u(P_6)$.

Also we have
\begin{multline}
|\bar{v}^{(6)}|\le |v_1^{(6)}|+|v_2^{(6)}|\le |\bar{v}^{(5)}|+2|\bar{v}^{(5)}| \left| \lambda^{2n}6C_0^2\mu(h)(1+O(n\lambda^{-2n}h))+O(n\lambda^{-2n}h)\right|\le \\ \le
2C_0^2|\bar{v}^{(1)}|+2\cdot 2C_0^2|\bar{v}^{(1)}|\left| \lambda^{2n}6C_0^2\mu(h)(1+O(n\lambda^{-2n}h))+O(n\lambda^{-2n}h)\right|\le 25C_0^4h^{-1-\nu}|\bar{v}^{(1)}|
\end{multline}
\end{proof}

\subsection{Markov partition and its thickness}

Standard arguments of the hyperbolic theory (see, for example, \cite{IL}) show that the Cone condition (Lemmas
\ref{conecondition1} and \ref{conecondition2}) together with the geometry of the map $\tilde{T}$ imply the
existence of the hyperbolic fixed point $\mathbf{Q}$ of the map $\tilde{T}$ in $\tilde{S}_1\cap
\tilde{T}(\tilde{S}_1)$. Our choice of the homoclinic points $q_h^u$ and $q_h^s$ implies that the eigenvalues of
$\mathbf{Q}$ are positive. Denote the heteroclinic point where $W^s_{loc}(\mathbf{Q})$ intersects $W^u(O)=\{(0,
x)|x\in \Bbb{R}\}$ by $(x_s, 0)$, and the heteroclinic point where $W^u_{loc}(\mathbf{Q})$ intersects
$W^s(O)=\{(y,0)|y\in \Bbb{R}\}$ by $(0, y_u)$.

Denote the segments of stable and unstable manifolds that connect the fixed points $O$ and $\mathbf{Q}$ with
these heteroclinic points by

$\gamma^u(O)$ -- connects $O$ and $(x_s, 0)$,\ \ \ \, $\gamma^s(O)$ -- connects $O$ and $(0, y_u)$;

$\gamma^u(\mathbf{Q})$ -- connects $\mathbf{Q}$ and $(0, y_u)$, \ \ \ $\gamma^s(\mathbf{Q})$ -- connects
$\mathbf{Q}$ and $(x_s, 0)$.

Notice that $\gamma^s(\mathbf{Q})\subset \tilde{S}_1$ and $\gamma^u(\mathbf{Q})\subset \tilde{T}(\tilde{S}_1)$.

Let $\mathbf{S}$ be the square formed by $\gamma^u(O)$, $\gamma^s(O)$, $\gamma^u(\mathbf{Q})$ and
$\gamma^s(\mathbf{Q})$, $\mathbf{S}\subset \tilde{S}$.

Now define $\mathbf{S}_0=\rho(\rho^{-1}(\mathbf{S})\cap N_h^{-1}\circ
\rho^{-1}(\mathbf{S}))\subset \tilde{S}_0$ and $\mathbf{S}_1=
\rho(S_1\cap \rho^{-1}(\mathbf{S}))\subset \tilde{S}_1$, see Fig \ref{fig.p20}.
Notice that one of the vertical edges of $\mathbf{S}_1$ is
$\gamma^s(\mathbf{Q})$ and another is an intersection of
$\mathbf{S}$ and a vertical edge of $\tilde{S}_1$, and therefore it
intersects $W^u(O)$ at the point
$\rho(N_h^{-n}(1,0))=\rho((\lambda^{-n}, 0))=(1, 0)$. Similarly,
$\tilde{T}(\mathbf{S}_1)$ has a vertical edge
$[1, y_u]\subset Oy$.

Define now $\mathbf{T}=\tilde{T}|_{\mathbf{S}}$. The maximal invariant set of $\mathbf{T}$ in $\mathbf{S}$,
$\Lambda=\cap_{n\in \Bbb{Z}}\mathbf{T}^{-n}(\mathbf{S})$, is a "horseshoe"-type basic set with Markov partition
$\mathcal{P}=\{\mathbf{S}_0, \mathbf{S}_1\}$. The map $\mathbf{T}:\mathbf{S}_0\cup\mathbf{S}_1\to \mathbf{S}$
belongs to class $\mathcal{F}$ (see definition \ref{classf}).

\begin{figure}
  \includegraphics[width=0.5\textwidth]{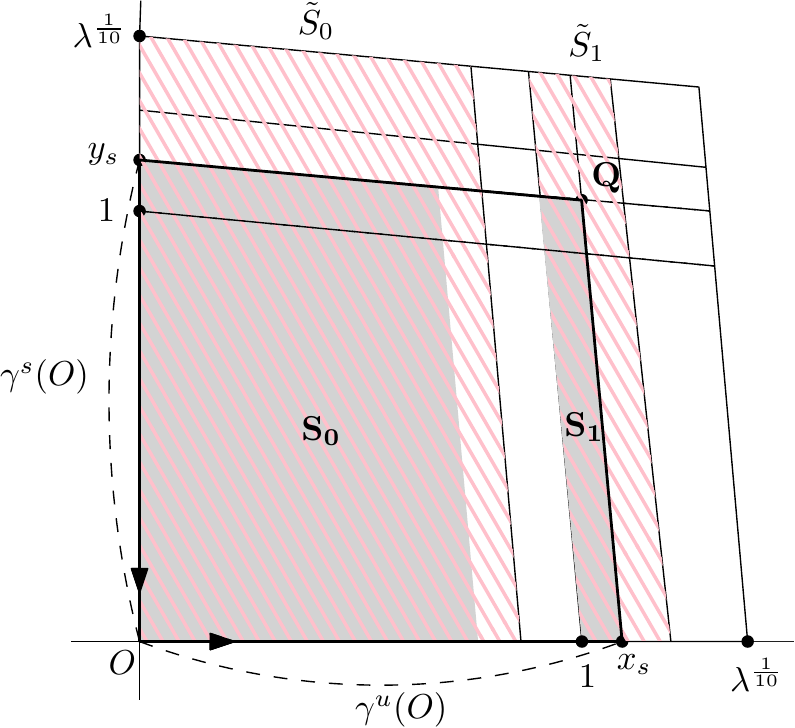}\\
  \caption{Rectangles $\tilde{S}_0$, $\tilde{S}_1$ and $\bf{S}_0$, $\bf{S}_1$. }\label{fig.p20}
\end{figure}

Consider now the Markov partitions
$$\mathcal{P}^s=\{[0, \lambda^{-1}x_s], [1, x_s]\}\ \ \ \ \text{\rm  and}\  \ \ \
\mathcal{P}^u=\{[0, \lambda^{-1}y_u], [1, y_u]\}$$ of the Cantor sets $K^s\subset Ox$ and
$K^u\subset Oy$ associated with the hyperbolic set $\Lambda$. We have
$$
\tau_L(\mathcal{P}^s)=\frac{\lambda^{-1}x_s}{1-\lambda^{-1}x_s}, \ \ \
\tau_R(\mathcal{P}^s)=\frac{x_s-1}{1-\lambda^{-1}x_s},
$$
$$
\tau_L(\mathcal{P}^u)=\frac{\lambda^{-1}y_u}{1-\lambda^{-1}y_u}, \ \ \
\tau_R(\mathcal{P}^u)=\frac{y_u-1}{1-\lambda^{-1}y_u}.
$$
\blm\label{distforpartition} The following estimates hold for all
$h\in (0, h_0)$ if $h_0$ is small enough: $$ \frac{1}{2}h^{-1}\le
\tau_L(\mathcal{P}^s)\le 2h^{-1} , \ \ \
0.01C_0^{-4}h^{\nu}\le \tau_R(\mathcal{P}^s)\le 250C_0^{4}h^{\nu},$$
$$
\frac{1}{2}h^{-1}\le \tau_L(\mathcal{P}^u)\le 2h^{-1} , \
\ \ 0.01C_0^{-4}h^{\nu}\le \tau_R(\mathcal{P}^u)\le
250C_0^{4}h^{\nu}.
$$ \elm

\begin{proof}[Proof of Lemma \ref{distforpartition}.] We will prove only
estimates for the partition  $\mathcal{P}^s$ (for $\mathcal{P}^u$
everything is the same). Notice first that
$1-\lambda^{-1}x_s\le x_s(1-\lambda^{-1})$.
Therefore
$$
\tau_L(\mathcal{P}^s)=\frac{\lambda^{-1}x_s}{1-\lambda^{-1}x_s}\ge
\frac{\lambda^{-1}x_s}{x_s-\lambda^{-1}x_s}=\frac{\lambda^{-1}}{1-\lambda^{-1}}=\frac{1}{\lambda-1}=\frac{1}{e^h-1}\ge
\frac{1}{2}h^{-1},
$$
since $e^h-1\le 2h$ for small $h$.

On the other hand, since $x_s\in (1, \lambda^{\frac{1}{10}})$
$$
\tau_L(\mathcal{P}^s)=\frac{\lambda^{-1}x_s}{1-\lambda^{-1}x_s}\le
\frac{\lambda^{-1}\lambda^{\frac{1}{10}}}{1-\lambda^{-1}\lambda^{\frac{1}{10}}}=\frac{1}{\lambda^{\frac{9}{10}}-1}=
\frac{1}{e^{\frac{9}{10}h}-1}\le \frac{10}{9}h^{-1}<{2}h^{-1}
$$

Now let us estimate $\tau_R(\mathcal{P}^s)$. Denote by $I$ the
segment of $W_{loc}^u(O)$ between the points
$(1, 0)$ and $(x_s, 0)$ (i.e. the bottom
horizontal edge of $\mathbf{S}_1$), $|I|=x_s-1$. Due to Lemma
\ref{conecondition2}
$$
0.01 C_0^{-4}h^{-1-\nu}|I|\le |\mathbf{T}(I)|\le 25C_0^4h^{-1-\nu}|I|.
$$
Since $\frac{1}{2}\le |\mathbf{T}(I)|\le 2 $ (this
follows from the Cone condition again), we have
$$
\frac{1}{50}C_0^{-4}h^{1+\nu}\le |I|\le
200C_0^4h^{1+\nu}
$$
Hence
$$
\tau_R(\mathcal{P}^s)\ge
\frac{\frac{1}{50}C_0^{-4}h^{1+\nu}}{1-\lambda^{-1}x_s}\ge
\frac{\frac{1}{50}C_0^{-4}h^{1+\nu}}{1-e^{-h}}\ge
0.01C_0^{-4}h^{\nu}
$$
since $1-e^{-h}\le h$ for $h\ge 0$, and
$$
\tau_R(\mathcal{P}^s)\le
\frac{200C_0^4h^{1+\nu}}{1-\lambda^{-1}x_s}\le \frac{200C_0^4h^{1+\nu}}{1-\lambda^{-1}\lambda^{\frac{1}{10}}}\le
\frac{200C_0^4h^{1+\nu}}{1-\lambda^{-\frac{9}{10}}}=\frac{200C_0^4h^{1+\nu}}{1-e^{-\frac{9}{10}h}}\le
250 C_0^4h^{\nu}.
$$
\end{proof}

\blm\label{property5} If $h_0>0$ is small enough then for all $h\in
(0, h_0)$ we have
$$
\text{\rm dist} (\mathbf{S}_0, \mathbf{S}_1)\ge 0.1 h \ \ \ \
\text{\rm and}\ \ \ \ \  \text{\rm dist} (\mathbf{T}(\mathbf{S}_0),
\mathbf{T}(\mathbf{S}_1))\ge 0.1 h.
$$
\elm

\begin{proof}[Proof of Lemma \ref{property5}.]
Notice that the vertical boundaries of $\mathbf{S}_0$ and
$\mathbf{S}_1$ are tangent to the cone field $\{K^u\}$. Consider the
left vertical edge of $\mathbf{S}_1$ and the right vertical edge of
$\mathbf{S}_0$. Their lowest points are $(1,0)$ and $(\lambda^{-1}x_s, 0)$, and the distance between them is
equal to
$$
1-\lambda^{-1}x_s\ge
1-\lambda^{-1}\lambda^{\frac{1}{10}}=1-\lambda^{-\frac{9}{10}}=1-e^{-\frac{9}{10}h}\ge
 \frac{9}{20}h
$$
if $h\in (0, h_0)$ and $h_0$ is small enough. From the cone
condition we have that the difference between $x$-coordinates of any
two points on those edges is greater than
$$
\frac{9}{20}h-2\cdot 100C_0^6h^{1+\nu}\ge
\frac{9}{20}h\left(1-\frac{4000}{9}C_0^6h^{\nu}\right)\ge
\frac{9}{40}h>0.1h
$$
if $h$ is small enough.

The proof of the second inequality is completely
similar.
\end{proof}

\subsection{Estimates of derivatives: verification of the conditions of Distortion Theorem}

We proved that the map $\mathbf{T}:\mathbf{S}_0\cup \mathbf{S}_1\to \mathbf{S}$ has an invariant locally maximal
hyperbolic set $\Lambda$ which is a two-component Smale horseshoe (i.e. $\mathbf{T}$ belongs to the class
$\mathcal{F}$) and obtained estimates of the lateral thicknesses of the corresponding Markov partitions. In
order to get estimates of the lateral thicknesses of the related Cantor sets we need to estimate the distortion
of the corresponding mappings.

Denote the differential of the map $\mathbf{T}:\mathbf{S}_0\cup \mathbf{S}_1\to \mathbf{S}$ by
 $D\mathbf{T}=\begin{pmatrix}
  a & b \\
   c & d \\
    \end{pmatrix}$, where $a, b, c$ and $d$ are smooth functions over $\mathbf{S}_0\cup \mathbf{S}_1$. Then the
    differential of the inverse map $\mathbf{T}^{-1}:\mathbf{T}(\mathbf{S}_0)\cup \mathbf{T}(\mathbf{S}_1)
    \to \mathbf{S}$ has the form $D\mathbf{T}^{-1}=\begin{pmatrix}
 \tilde{d} & -\tilde{b} \\
   -\tilde{c} & \tilde{a} \\
    \end{pmatrix}$, where $\tilde{a}=a\circ \mathbf{T}^{-1}$, $\tilde{b}=b\circ \mathbf{T}^{-1}$,
    $\tilde{c}=c\circ \mathbf{T}^{-1}$ and $\tilde{d}=d\circ
    \mathbf{T}^{-1}$. Notice that this notation agrees with the
    notation of Definition \ref{classofmaps}.

 \blm\label{estimates of derivatives s0} Consider the restriction of the map $\mathbf{T}$ to the
 rectangle $\mathbf{S}_0$.
There exists a constant $C_1>1$ (independent of $h$) such that

(1) $\lambda(1-C_1h\lambda^{-2n})\le |a|\le \lambda(1+C_1h\lambda^{-2n})$,

(2) $|d|\le \lambda^{-1}(1+C_1h\lambda^{-2n})$,

(3) $|b|, |c|\le C_1h\lambda^{-2n}$,

(4) $\left|\frac{\partial a}{\partial y}\right|, \left|\frac{\partial a}{\partial x}\right|,  \left|\frac{\partial b}{\partial x}\right|,
\left|\frac{\partial b}{\partial y}\right|, \left|\frac{\partial
c}{\partial x}\right|, \left|\frac{\partial c}{\partial y}\right|,\left|\frac{\partial d}{\partial y}\right|, \left|\frac{\partial d}{\partial x}\right|
\le C_1h\lambda^{-2n}$,

(5) $\left|\frac{\partial \tilde a}{\partial y}\right|, \left|\frac{\partial \tilde a}{\partial x}\right|,  \left|\frac{\partial \tilde b}{\partial x}\right|,
\left|\frac{\partial \tilde b}{\partial y}\right|, \left|\frac{\partial
\tilde c}{\partial x}\right|, \left|\frac{\partial \tilde c}{\partial y}\right|,\left|\frac{\partial \tilde d}{\partial y}\right|, \left|\frac{\partial \tilde d}{\partial x}\right|
\le C_1h\lambda^{-2n}$,
 \elm
 \begin{proof}[Proof of Lemma \ref{estimates of derivatives s0}]
Notice that Lemma \ref{estimates of derivatives s0} is a stronger version of Lemma \ref{l.Dt}.  Since
 $$
\mathbf{T}|_{\mathbf{S}_0}(x,y)=\rho\circ N_h\circ \rho^{-1}(x,y)=(\Delta_h(t(xy))x, \Delta^{-1}_h(t(xy))y),
$$
we have
\beq
\begin{aligned}
& a(x,y)= \Delta_h(t(xy))+xy\Delta_h'(t(xy))t'(xy)             \\
& b(x,y)= x^2\Delta_h'(t(xy))t'(xy)             \\
& c(x,y)=    -\frac{y^2}{\Delta_h^2(t(xy))}\Delta_h'(t(xy))t'(xy)          \\
& d(x,y)=     \Delta_h^{-1}(t(xy))-xyt'(xy)\frac{\Delta_h'(t(xy))}{\Delta_h^2(t(xy))}
\end{aligned}
\eneq
 Now the required estimated follow by a direct calculation from Lemmas \ref{analogduarte} and \ref{l.t}.
 \end{proof}

 \blm\label{variations0} The variation of $\log |a(x,y)|$ in
$\mathbf{S}_0$ is less than $4C_1h\lambda^{-2n}$. \elm
\begin{proof}[Proof of Lemma \ref{variations0}]
This follows immediately from the estimate (1) in Lemma \ref{estimates of derivatives s0}.
 \end{proof}

 \blm\label{estimates of derivatives} Consider the restriction of the map $\mathbf{T}$ to the
 rectangle $\mathbf{S}_1$.
There exists a constant $C_1>1$ (independent of $h$) such that

(1) $|d|\le \lambda^{-2n}C_1$, $|b|, |c|\le C_1$,

(2) $C_1^{-1}h^{-1-\nu}\le |a|\le C_1h^{-1-\nu}$,

(3) $\left|\frac{\partial b}{\partial x}\right|,
\left|\frac{\partial b}{\partial y}\right|, \left|\frac{\partial
c}{\partial x}\right|, \left|\frac{\partial c}{\partial y}\right|,
\left|\frac{\partial a}{\partial y}\right|\le C_1$,

(4) $\left|\frac{\partial d}{\partial x}\right|,\left|\frac{\partial d}{\partial y}\right|\le \lambda^{-2n}C_1,
$

(5) $\left|\frac{\partial a}{\partial x}\right|\le C_1h^{-2-\nu}, $

(6) $\left|\frac{\partial \tilde{b}}{\partial x}\right|,
\left|\frac{\partial \tilde{b}}{\partial y}\right|,
\left|\frac{\partial \tilde{c}}{\partial x}\right|,
\left|\frac{\partial \tilde{c}}{\partial y}\right|,
\left|\frac{\partial \tilde{a}}{\partial x}\right|\le C_1$,

(7) $\left|\frac{\partial \tilde{d}}{\partial y}\right|,\left|\frac{\partial \tilde{d}}{\partial
x}\right|\le
\lambda^{-2n}C_1, $

(8) $\left|\frac{\partial \tilde{a}}{\partial y}\right|\le
C_1h^{-2-\nu}. $
 \elm
\begin{proof}[Proof
of Lemma \ref{estimates of derivatives}.]
We need to study the differential of the map
$$
\mathbf{T}|_{\mathbf{S}_1}=\widehat{G}\circ C_h\circ \Psi_h\circ H^{k(h)}\circ  \Psi_h^{-1}\circ C_h^{-1}\circ G, \ \text{ where} \ \ G:\mathbf{S}_1\to R^u, \ \widehat{G}: R^s \to\mathbf{S}_1.
$$
In order to distinguish the points from $R^u$ and from $R^s$ let us
denote the coordinates in $R^u$ by $(\mathbf{x}, \mathbf{y})$ and
the coordinates in $R^s$ by $(\mathbf{X}, \mathbf{Y})$. Then the map
$C_h\circ \Psi_h\circ H^{k(h)}\circ \Psi_h^{-1}\circ C_h^{-1}:R^u\to
U$ can be represented as a composition
$$
(\mathbf{x}, \mathbf{y})\mapsto (t(\mathbf{x}, \mathbf{y}),
E(\mathbf{x}, \mathbf{y}))\mapsto (\mathbf{X}(t,E),
\mathbf{Y}(t,E)), \ \ \ (t,E)\in \Pi_{r_0, E_0},
$$
where $$(t(\mathbf{x}, \mathbf{y}), E(\mathbf{x},
\mathbf{y}))=H^{k(h)}\circ \Psi_h^{-1}\circ C_h^{-1}(\mathbf{x},
\mathbf{y})\ \ \  \text{\rm  and}\ \ \  (\mathbf{X}(t,E),
\mathbf{Y}(t,E))=C_h\circ \Psi_h(t,E).$$

Let us denote
$$
\begin{pmatrix}
  a_0 & b_0 \\
  c_0 & d_0 \\
\end{pmatrix}=D(C_h\circ \Psi_h\circ H^{h(h)}\circ
\Psi^{-1}_h\circ C_h^{-1}),
$$
then $$D\mathbf{T}|_{\mathbf{S}_1}(x,y)=\begin{pmatrix}
  a(x,y) & b(x,y) \\
  c(x,y) & d(x,y) \\
\end{pmatrix}=D\widehat{G}(\mathbf{X},\mathbf{Y})\begin{pmatrix}
  a_0(\mathbf{x}, \mathbf{y}) & b_0(\mathbf{x}, \mathbf{y}) \\
  c_0(\mathbf{x}, \mathbf{y}) & d_0(\mathbf{x}, \mathbf{y}) \\
\end{pmatrix}D{G}(x,y), 
$$
where
$$
DG(x,y)=\begin{pmatrix}
  1 & 0 \\
  \frac{2ny^2\Delta_h'(t(xy))t'(xy)}{\Delta_h^{2n+1}(t(xy))} & \Delta_h^{-2n}(t)-\frac{2nxy\Delta'_h(t(xy))t'(xy)}{\Delta_h^{2n+1}(t(xy))} \\
\end{pmatrix},$$
$$ D\widehat{G}(\mathbf{X},\mathbf{Y})=\begin{pmatrix}
           \Delta_h^{2n}(\mathbf{X}\mathbf{Y})+2n\mathbf{X}\mathbf{Y}\Delta_h^{2n-1}(\mathbf{X}\mathbf{Y})\Delta_h'(\mathbf{X}\mathbf{Y}) & 2n\mathbf{X}^2\Delta_h^{2n-1}(\mathbf{X}\mathbf{Y})\Delta_h'(\mathbf{X}\mathbf{Y}) \\
           0 & 1 \\
         \end{pmatrix}.
$$
Therefore
\begin{align*}
    a(x,y)=& (\Delta_h^{2n}(\mathbf{X}\mathbf{Y})+2n\mathbf{X}\mathbf{Y}\Delta_h^{2n-1}  (\mathbf{X}\mathbf{Y})\Delta_h'(\mathbf{X}\mathbf{Y}))  \left(a_0(\mathbf{x}, \mathbf{y})+\frac{2b_0(\mathbf{x}, \mathbf{y})ny^2\Delta_h'(t(xy))t'(xy)}{\Delta_h^{2n+1}(t(xy))}\right)+ \\
& \qquad \quad \qquad +  2n\mathbf{X}^2\Delta_h^{2n-1}  (\mathbf{X}\mathbf{Y})\Delta_h'(\mathbf{X}\mathbf{Y})  \left(c_0(\mathbf{x}, \mathbf{y})+\frac{2d_0(\mathbf{x}, \mathbf{y})ny^2\Delta_h'(t(xy))t'(xy)}{\Delta_h^{2n+1}(t(xy))}\right), \\
    b(x,y)=& (\Delta_h^{2n}(\mathbf{X}\mathbf{Y})+  2n  \mathbf{X}\mathbf{Y}\Delta_h^{2n-1}  (\mathbf{X}\mathbf{Y}) \Delta_h'(\mathbf{X}\mathbf{Y}))  \left(\Delta_h^{-2n}(t(xy))b_0(\mathbf{x}, \mathbf{y}) -\frac{2b_0(\mathbf{x}, \mathbf{y})nxy\Delta'_h(t(xy))t'(xy)}{\Delta_h^{2n+1}(t(xy))}\right) \\
& \qquad \quad \qquad +  2n\mathbf{X}^2\Delta_h^{2n-1}  (\mathbf{X}\mathbf{Y})\Delta_h'(\mathbf{X}\mathbf{Y})  \left(\Delta_h^{-2n}(t(xy))d_0(\mathbf{x}, \mathbf{y})-\frac{2d_0(\mathbf{x}, \mathbf{y})nxy\Delta'_h(t(xy))t'(xy)}{\Delta_h^{2n+1}(t(xy))}\right),\\
c(x,y)=& c_0(\mathbf{x}, \mathbf{y})+ \frac{2d_0(\mathbf{x}, \mathbf{y})ny^2\Delta_h'(t(xy))t'(xy)}{\Delta_h^{2n+1}(t(xy))},\\
d(x,y)=& \Delta_h^{-2n}(t(xy))d_0(\mathbf{x}, \mathbf{y})- \frac{2d_0(\mathbf{x}, \mathbf{y})nxy\Delta'_h(t(xy))t'(xy)}{\Delta_h^{2n+1}(t(xy))}.
\end{align*}

The $C^3$-norm of the map $C_h\circ \Psi_h\circ H^{k(h)}\circ
\Psi^{-1}_h\circ C_h^{-1}$ (as well as of its inverse) is uniformly
bounded by some constant independent of $h$. Together with Lemmas \ref{analogduarte} and \ref{l.t} this implies that
 $d(x,y)=O(\lambda^{-2n})$, $b(x,y)=O(1), c(x,y)=O(1)$, and this proves the statement (1). Inequality (2) follows directly from Lemma \ref{conecondition2} for large enough $C_1$, for example, $C_1> 100C_0^4$. Using the estimates
 $$
 \mathbf{X},\mathbf{y},\frac{\partial\mathbf{X}}{\partial y}, \frac{\partial\mathbf{Y}}{\partial y}=O(\lambda^{-2n}),\  \frac{\partial(\mathbf{XY})}{\partial x}, \frac{\partial\mathbf{X}}{\partial x}=O(1), \ \frac{\partial(\mathbf{XY})}{\partial y}=O(\lambda^{-2n}),
 $$
 $$
 \frac{\partial\mathbf{x}}{\partial x}=1,\ \frac{\partial\mathbf{x}}{\partial y}=0,\ \frac{\partial\mathbf{y}}{\partial x}=O(\lambda^{-4n}),\ \frac{\partial\mathbf{y}}{\partial y}=O(\lambda^{-2n}),\ n=O(h^{-2}),
 $$
 one can show that partial derivatives of $b(x,y)$ and $c(x,y)$ are bounded, and that $\frac{\partial d}{\partial x},\frac{\partial d}{\partial y}$ are of order $O(\lambda^{-2n})$ (we omit the details of these tedious but straightforward calculations).

   In order to study the partial derivatives of $a(x,y)$, let us first obtain some estimates for $a_0(\mathbf{x},\mathbf{y})$.
    We have
$$
a_0(\mathbf{x}, \mathbf{y})=\frac{\partial}{\partial
\mathbf{x}}\mathbf{X}(t(\mathbf{x}, \mathbf{y}), E(\mathbf{x},
\mathbf{y}))\ \ \ \text{\rm and }\ \ \
\frac{\partial a_0}{\partial \mathbf{x}}=\frac{\partial^2}{\partial
\mathbf{x}^2}\mathbf{X}(t(\mathbf{x}, \mathbf{y}), E(\mathbf{x},
\mathbf{y})).
$$
In particular, since $E(\mathbf{x}, 0)=0$,
\begin{multline}\label{partialazero}
    \frac{\partial a_0}{\partial \mathbf{x}}(\mathbf{x}, 0)=\frac{d^2}{d
\mathbf{x}^2}\mathbf{X}(t(\mathbf{x}, 0),
0)=\frac{d}{d\mathbf{x}}\left(\frac{\partial \mathbf{X}}{\partial
t}(t(\mathbf{x}, 0), 0)\cdot \frac{\partial t}{\partial \mathbf{x}}(\mathbf{x},0)\right)=\\
=\frac{\partial^2\mathbf{X}}{\partial t^2}(t(\mathbf{x}, 0), 0)\cdot
\left(\frac{\partial t}{\partial
\mathbf{x}}(\mathbf{x},0)\right)^2+\frac{\partial
\mathbf{X}}{\partial t}(t(\mathbf{x}, 0), 0)\cdot \frac{\partial^2
t}{\partial \mathbf{x}^2}(\mathbf{x}, 0)=\\
=\frac{\partial^2\mathbf{X}}{\partial t^2}(t(\mathbf{x}, 0), 0)\cdot
\left(\frac{\partial t}{\partial
\mathbf{x}}(\mathbf{x},0)\right)^2+a_0(\mathbf{x}, 0)\cdot
\left(\frac{\partial t}{\partial
\mathbf{x}}(\mathbf{x},0)\right)^{-1}\cdot \frac{\partial^2
t}{\partial \mathbf{x}^2}(\mathbf{x}, 0)
\end{multline}
From the Cone condition (more precisely, from Steps 3-5 of the proof
of Lemma \ref{conecondition2}) we know that $a_0(\mathbf{x},0)=O(\mu(h))$. Also since $C^3$-norms of maps $C_h\circ \Psi_h$ and
$H^{k(h)}\circ \Psi_h^{-1}\circ C_h^{-1}$ are bounded by $C_0$, we
have
$$
\left|\frac{\partial a_0}{\partial \mathbf{x}}(\mathbf{x}, 0)\right|\le O\left(\frac{\partial^2\mathbf{X}}{\partial t^2}(t(\mathbf{x}, 0),
0)\right) + O(\mu(h)).
$$
Now we need to estimate $\left|\frac{\partial^2\mathbf{X}}{\partial t^2}(t(\mathbf{x}, 0), 0)\right|$. Notice
that the image of the $Oy$ ax under the map $C_h\circ \Psi_h$ is a graph of the function $E=\Theta(t)$, and
therefore $\mathbf{X}(t, \Theta(t))=0$. This implies that
$$
\frac{d}{dt}(\mathbf{X}(t, \Theta (t))=0=\frac{\partial \mathbf{X}}{\partial t}(t, \Theta
(t))+\frac{\partial \mathbf{X}}{\partial E}(t, \Theta (t))\cdot {\dot{\Theta}(t)}
$$
and
\begin{multline}\label{secondderX}
    \frac{d^2}{dt^2}(\mathbf{X}(t, \Theta (t))=0=\\
    =\frac{\partial^2
\mathbf{X}}{\partial t^2}(t, \Theta (t))+2\frac{\partial^2\mathbf{X}}{\partial t\partial E}(t, \Theta (t))
{\dot{\Theta}(t)}+\frac{\partial^2 \mathbf{X}}{\partial E^2}(t, \Theta (t)) ({\dot{\Theta}(t)})^2+
\frac{\partial \mathbf{X}}{\partial E}(t, \Theta (t)) {\ddot{\Theta}(t)}
\end{multline}
Since for $\Theta(t)$,  $\dot\Theta(t)$ and $\ddot\Theta(t)$ we have asymptotics (\ref{thetastar}) (see
Theorem \ref{t.Gelfreich} and Remark \ref{remarkaboutmu}), we get
$$
\left|\frac{\partial^2 \mathbf{X}}{\partial t^2}(t, \Theta (t))\right|\le 2C_0\cdot
2\mu(h)+C_0(2\mu(h))^2+C_0\cdot 4\pi h^{-1}\mu(h) = O(h^{-1}\mu(h))
$$
if $h$ is small enough.

At the same time by the mean value theorem we have
$$
\left|\frac{\partial^2 \mathbf{X}}{\partial t^2}(t, 0)\right|\le \left|\frac{\partial^2 \mathbf{X}}{\partial
t^2}(t, \Theta (t))\right|+C_0|\Theta(t)|<20C_0h^{-1}\mu(h)+C_0\frac{1}{\pi}h\mu(h)= O(h^{-1}\mu(h))
$$
Finally we have
$$
\frac{\partial a_0}{\partial \mathbf{x}}(\mathbf{x}, 0)= O(h^{-1}\mu(h)) + O(\mu(h))= O(h^{-1}\mu(h)),
$$
and since $\mathbf{y}=O(\lambda^{-2n})=O(h^{1+\nu}\mu(h))$, by the mean value theorem we also have $
\frac{\partial a_0}{\partial \mathbf{x}}(\mathbf{x}, \mathbf{y})= O(h^{-1}\mu(h)).
$
Now we have
$$
\frac{\partial a}{\partial y}(x,y)=a_0(\mathbf{x}, \mathbf{y})2n\Delta_h^{2n-1}(\mathbf{X}\mathbf{Y})\Delta_h'(\mathbf{X}\mathbf{Y})\frac{\partial (\mathbf{X}\mathbf{Y})}{\partial y}+\frac{\partial a_0}{\partial y}(\mathbf{x}, \mathbf{y})\Delta_h^{2n}(\mathbf{X}\mathbf{Y})+O(nh\lambda^{-2n}),
$$
and since $a_0(\mathbf{x}, \mathbf{y})=O(\mu(h))$ and $\frac{\partial a_0}{\partial y}(\mathbf{x}, \mathbf{y})=O(\lambda^{-2n}),$ we also have
$$
\frac{\partial a}{\partial y}(x,y)=O(n\mu(h)\lambda^{2n}h\lambda^{-2n})+O(1)+O(nh\lambda^{-2n})=O(1).
$$
Let us now estimate $\frac{\partial a}{\partial x}(x,y)$. We have
\begin{multline*}
\frac{\partial a}{\partial x}(x,y)=\frac{\partial a_0}{\partial x}(\mathbf{x}, \mathbf{y})(\Delta_h^{2n}(\mathbf{X}\mathbf{Y})+2n\mathbf{X}\mathbf{Y}\Delta_h^{2n-1}  (\mathbf{X}\mathbf{Y})\Delta_h'(\mathbf{X}\mathbf{Y})) +\\ +a_0(\mathbf{x}, \mathbf{y})2n\frac{\partial (\mathbf{X}\mathbf{Y})}{\partial x}\Delta_h^{2n-1}(\mathbf{X}\mathbf{Y})\Delta_h'(\mathbf{X}\mathbf{Y})+O(nh).
\end{multline*}
Therefore
\begin{multline*}
\frac{\partial a}{\partial x}(x,y)=O(h^{-1}\mu(h))\left(O(\lambda^{2n})+O(n\lambda^{-2n}\lambda^{2n}h)\right)+O(\mu(h)n\lambda^{2n}h)+O(nh)=\\
=O(h^{-1}h^{-1-\nu})+O(h^{-1-\nu}nh)+O(nh)=O(h^{-2-\nu}).
\end{multline*}
This proves statements (3), (4), and (5). Estimates (6), (7), and (8) are symmetric to the estimates (3), (4), and (5).
\end{proof}

\blm\label{variation} The variation of $\log |a(x,y)|$ in
$\mathbf{S}_1$ is less than $600C_1^2C_0^6h^{\nu}$. \elm
\begin{proof}[Proof
of Lemma \ref{variation}.] Take two points $(x_1, y_1)$ and $(x_2,
y_2)$ from $\mathbf{S}_1$.  We want to estimate $|\log a(x_1,
y_1)-\log a(x_2, y_2)|$ by using the mean value theorem. Generally
speaking, the set $\mathbf{S}_1$ is not convex, so we need some
preparations to apply it.

 Let $\tilde{\gamma}$ be the intersection $\tilde{\gamma}=\mathbf{T}(\mathbf{S}_1)\cap \{x=\frac{1}{2}\}$.
 Then $\hat{\gamma}=\mathbf{T}^{-1}(\tilde{\gamma})$ is a smooth
 curve tangent to the cone field $\{K^s\}$, $\hat{\gamma}\subset
 \mathbf{S}_1$.
Denote $\hat{x}_1=\{y=y_1\}\cap \hat{\gamma}$ and
$\hat{x}_2=\{y=y_2\}\cap \hat{\gamma}$. Notice that the whole
interval with the end points $(x_1, y_1)$ and $(\hat{x}_1, y_1)$
belongs to $\mathbf{S}_1$, as well as the interval with end points
$(x_2, y_2)$ and $(\hat{x}_2, y_2)$. Now we have
\begin{multline}
  |\log a(x_1, y_1)-\log a(x_2, y_2)|\le \\
  \le |\log a(x_1, y_1)-\log a(\hat{x}_1, y_1)|+|\log a(\hat{x}_1, y_1)-\log a(\hat{x}_2,
  y_2)|+|\log a(\hat{x}_2, y_2)-\log a(x_2, y_2)|
\end{multline}
Due to the Cone condition the width of $\mathbf{S}_1$ is not greater
than $200 C_0^4h^{1+\nu}$. By the mean value theorem we have
\begin{multline}\label{est}
    |\log a(x_1, y_1)-\log a(\hat{x}_1, y_1)|\le \\
    \le  \frac{1}{|a(x_1^*,
y_1)|}\left|\frac{\partial a}{\partial x}(x_1^*,
y_1)\right||x_1-\hat{x}_1|\le C_1h^{1+\nu}\cdot C_1h^{-2-\nu}\cdot
200C_0^4h^{1+\nu}=200C_1^2C_0^4h^{\nu}
\end{multline}
Similarly
$$
|\log a(\hat{x}_2, y_2)-\log a({x}_2, y_2)|\le 200 C_1^2C_0^4h^{\nu}
$$
Now parameterize the curve $\hat{\gamma}$ by the parameter $y$,
$\hat{\gamma}=\hat{\gamma}(x(y), y), y\in [y_1, y_2]$ (or $y\in
[y_2, y_1]$ if $y_2<y_1$). Consider a function $g(y)=\log
a(\hat{\gamma}(x(y), y)$. Since $\hat{\gamma}$ is tangent to the
cone field $\{K^s\}$, for some $y^*\in [y_1, y_2]$ we have
\begin{multline}\label{est1}
    |g(y_1)-g(y_2)|=|g'(y^*)|\cdot |y_1-y_2|=\frac{1}{|a(\hat{\gamma}(x(y^*),
    y^*)|}\cdot \left|\frac{\partial a}{\partial x}\hat{\gamma}'_x+\frac{\partial a}{\partial
    y}\hat{\gamma}'_y\right|\cdot |y_1-y_2|\le \\
    \le C_1h^{1+\nu}\left(\left|\frac{\partial a}{\partial x}\right||\hat{\gamma}'_x|+
    \left|\frac{\partial a}{\partial
    y}\right||\hat{\gamma}'_y|\right)\le
    C_1h^{1+\nu}(C_1h^{-2-\nu}\cdot 100C_0^6h^{1+\nu}+C_1)\le\\
    \le
    C_1h^{\nu}(100C_1C_0^6+C_1h)\le 200 C_1^2C_0^6h^{\nu}
\end{multline}
Finally we have
$$
|\log a(x_1, y_1)-\log a(x_2, y_2)|\le 400 C_1^2C_0^4h^{\nu}+200
C_1^2C_0^6h^{\nu}<600C_1^2C_0^6h^{\nu}
$$
\end{proof}

The following Proposition directly follows from Lemmas
\ref{conecondition1}, \ref{conecondition2}, \ref{property5}, \ref{estimates of derivatives s0}, \ref{variations0},
\ref{estimates of derivatives},    and \ref{variation}.
\bprop\label{tbelongstoclassf} The map $\mathbf{T}:\mathbf{S}_0\cup
\mathbf{S}_1\to \mathbf{S}$ belongs to the class $\mathcal{F}(C^*,
\gamma, \varepsilon)$ (see definition \ref{classofmaps}), where
$C^*=120C_1^4C_0^6$, $\gamma=1200C_1^3C_0^6h^{\nu}$ and
$\varepsilon=120C_1^3C_0^6h^{1+\nu}$. \enprop

\subsection{Final step}

\begin{proof}[Proof of Theorem \ref{t.henonfamily}.] Properties 1. and 2. of Theorem \ref{t.henonfamily} clearly follow
from the construction and the Cone condition. Let us combine now
Proposition \ref{tbelongstoclassf} with Duarte Distortion Theorem
(Theorem \ref{distortiontheorem}) and Lemma \ref{distforpartition}.
Let us assume that $h$ is small enough so that $e^{D(C^*,
\varepsilon, \gamma)}<2$. Then we have
$$
\frac{1}{4}h^{-1}\le \tau_L(K^s)\le 4h^{-1}, \ \ \
\frac{1}{200}C_0^{-4}h^{\nu}\le \tau_R(K^s)\le 500 C_0^4h^{\nu},
$$
$$
\frac{1}{4}h^{-1}\le \tau_L(K^u)\le 4h^{-1}, \ \ \
\frac{1}{200}C_0^{-4}h^{\nu}\le \tau_R(K^u)\le 500 C_0^4h^{\nu}.
$$
Therefore $$\tau_L(K^s)\tau_R(K^s)\ge
\frac{1}{800}C_0^{-4}h^{-1+\nu}\to \infty \ \ \ \text{\rm  as}  \ \
\ h\to 0 \ \ \ \text{\rm  (i.e. $a\to -1$}).$$ Similarly
$\tau_L(K^u)\tau_R(K^u)\to \infty$ as $a\to -1$, so
$\tau_{LR}(\Lambda)\to \infty$ as $a\to -1$. Notice that this implies that Lemma \ref{l.leftright} can be applied to the Cantor sets $K^s$ and $K^u$, and this is how Duarte proved existence of the conservative Newhouse phenomena in \cite{Du1, Du2}.

To check the property 4., we apply Proposition \ref{logestimate}
(notice that we are exactly in the setting of Remark \ref{remark5}).
We have
\begin{multline}
  \text{dim}_HK^s\ge \frac{\log\left(1+\frac{\tau_L(K^s)}{1+\tau_R(K^s)}\right)}{\log\left(1+\frac{1+\tau_L(K^s)}{\tau_R(K^s)}\right)}
  \ge  \frac{\log\left(1+\frac{\frac{1}{4}h^{-1}}{1+500 C_0^4h^{\nu}}\right)}
  {\log\left(1+\frac{1+4h^{-1}}{\frac{1}{200}C_0^{-4}h^{\nu}}\right)}=\\
  =
  \frac{\log \left[h^{-1}\left(h+\frac{1}{4+2000C_0^4h^{\nu}}\right)\right]}
  {\log
  \left[h^{-1-\nu}\left(h^{1+\nu}+200C_0^4(4+h)\right)\right]}=\frac{1-O((\log h)^{-1})}{1+\nu-O((\log
  h)^{-1})}>\frac{1}{1+2\nu}
\end{multline}
if $h$ is small enough. Since $\nu$ could be chosen arbitrary small,
$\text{dim}_HK^s\to 1$ as $h\to 0$ (i.e. $a\to -1$). Similarly
$\text{dim}_HK^u\to 1$ as  $a\to -1$. Since
$\text{dim}_H{\Lambda}=\text{dim}_HK^s+\text{dim}_HK^u$ (\cite{MM},
see also \cite{PV}), we have
$$
\text{dim}_H{\Lambda}\to 2\ \ \text{\rm as $a\to -1$}.
$$
Theorem \ref{t.henonfamily} is proved.
\end{proof}

\section{Conservative homoclinic bifurcations and hyperbolic sets of large Hausdorff dimension: the proof}\label{s.consproof}

Here we derive Theorem \ref{t.maintheorem} from Theorem \ref{t.henonfamily}.
\begin{proof}[Proof of Theorem \ref{t.maintheorem}.] First of all, Theorem B from \cite{Du4} claims that a generic unforlding of a conservative homoclinic tangency leads to appearance of $C^2$-stably-wild hyperbolic basic set. More precisely, there exists an open set $\mathcal{U}_0\subset \mathbb{R}^1, 0\in \overline{\mathcal{U}_0},$ such that
\begin{itemize}
\item each map $f_{\mu}, \mu\in \mathcal{U}_0$, has a basic set $\Lambda_{\mu}^*$ exhibiting homoclinic tangencies;


\item $P_{\mu}\in \Lambda_{\mu}^*$, where $P_{\mu}$ is a continuation of the saddle $P_0$;

\item there exists a dense subset $D_0\subset \mathcal{U}_0$ such that for each $\mu\in D_0$ the saddle $P_{\mu}$ has a quadratic homolcinic tangency which unfolds generically with $\mu$.
\end{itemize}

Choose a sequence of parameter values $\{\mu_n\}_{n\in \mathbb{N}}\subset D_0$ dense in $\mathcal{U}_0$. Fix any $a\in \mathbb{R}$. The renormalization technics by Mora-Romero \cite{MR1} prove that an appropriately chosen and rescaled map near
a homoclinic tangency is $C^r$--close to a Henon map $H_a$. Namely, the following statement holds.

\bthm[\cite{MR1}, based on \cite{AS, gs3}]
Let $\{f_{\nu}\}\subset \text{\rm Diff}^{\infty}(M^2, \omega)$ be a smooth family of area preserving maps
unfolding generically a quadratic homoclinic tangency at the point $Q_0\in M$ and parameter $\nu=0$. Then there are, for all large enough $n\in \mathbb{N}$, reparametrizations $\nu=\nu_n(a)$ of the parameter variable $\nu$ and $a$-dependent coordinates
$$
(x,y)\mapsto \Psi_{n,a}(x,y)\in M^2
$$
such that

(1) for each compact $K$, in the $(a,x,y)$-space, the images of $K$ under the maps
$$
(a,x,y)\mapsto (\nu_n(a), \Psi_{n,a}(x,y))
$$
converge to $(0, Q_0)\in \mathbb{R}\times M^2$, as $n\to \infty$.

(2) the domains of the maps
$$
(a, x, y)\mapsto (a, \Psi^{-1}_{n,a}\circ f^n_{\nu_n(a)}\circ \Psi_{n,a}(x,y))
$$
converge to $\mathbb{R}^3$ as $n\to \infty$ and the maps convege in the $C^{\infty}$ topology to the conservative Henon map
$$
(a, x, y)\mapsto (a, y, -x+a-y^2).
$$
\ethm

 By Theorem \ref{t.henonfamily} for $a$ slightly
larger than $-1$ the map $H_a$ has an invariant hyperbolic set $\Lambda_a$ of Hausdorff dimension close to $2$ with persistent
hyperbolic tangencies. By continuous dependence of Hausdorff
dimension of an invariant hyperbolic set on a diffeomorphism
\cite{MM,PV} near each $\mu_n$ there is an open interval of
parameters $U_n\subset \mathcal{U}_0$ such that for $\mu\in U_n$ the map $f_\mu$ has an
invariant locally maximal transitive hyperbolic set $\Delta_{\mu}^*$ with Hausdorff
dimension greater than $2-\delta$. Set $\mathcal{U}=\cup_{n\in \mathbb{N}}U_n$ and $D=D_0\cap \mathcal{U}$. The hyperbolic saddle $P_\mu$ and the set $\Delta_{\mu}^*$ are homoclinically related, see Lemma 2 from \cite{Du1}. Therefore
for every $\mu\in U_n$ there exists a basic set $\Delta_{\mu}$ such
that $P_\mu\in \Delta_\mu$ and $\Lambda^*_\mu\cup\Delta^*_\mu\subset \Delta_\mu$.
Since $\Lambda_{\mu}^*$ has persistent homoclinic tangencies, so does
$\Delta_{\mu}$. Also, $\text{\rm dim}_H \Delta_{\mu}\ge \text{\rm
dim}_H \Delta^*_{\mu}>2-\delta$. By construction the parts (1) and (2) of Theorem \ref{t.maintheorem} are now satisfied.

Let us now observe how elliptic periodic points appear. Take any
$\mu\in \mathcal{U}$. If $Q_{\mu}$ is a transversal homoclinic point of the
saddle $P_{\mu}$ then in can be continued for some intervals of
parameters $I_Q\subseteq \mathcal{U}$. Assume that $I_Q\subseteq \mathcal{U}$ is a
maximal subinterval of $\mathcal{U}$ where such a continuation is possible.
All homoclinic points of $P_{\mu}$ for all values $\mu\in \mathcal{U}$
generate countable number of such subintervals $\{I_s\}_{s\in
\mathbb{N}}$ in $\mathcal{U}$.

From \cite{MR1} it follows that for each $I_s$ there exists a
residual set ${R}_s\subseteq  I_s$ of parameters such that for
$\mu\in R_s$ the corresponding homoclinic point $Q_{\mu}$ is an
accumulation point of elliptic periodic points of $f_{\mu}$. Denote
$\widetilde{{R}}_s=\left(\mathcal{U}\backslash \overline{I_s}\right)\cup
R_s$ -- residual subset of $\mathcal{U}$. Now set $\mathcal{R}_1=\cap_{s\in
\mathbb{N}}\widetilde{R}_s$ -- also a residual subset in $\mathcal{U}$. For
$\mu\in \mathcal{R}_1$ every transversal homoclinic point of the
saddle $P_{\mu}$ is an accumulation point of elliptic periodic
points of $f_{\mu}$, and  this proves (3.1).

Now let us see that for a residual set of parameters in $\mathcal{U}$ the homoclinic class of $P_{\mu}$ has full Hausdorff
dimension. In the same way as we constructed $\mathcal{U}$ starting with $\mathcal{U}_0$, from Theorem \ref{t.henonfamily} and \cite{MR1} it follows that for every
$m\in \Bbb{N}$ there exists an open and dense subset $\mathcal{A}_m\subset
\mathcal{U}$ such that for every $\mu\in \mathcal{A}_m$ there exists a hyperbolic set
$\Delta^m_\mu$ such that $\text{\rm dim}_H\Delta^m_\mu
> 2-\frac{1}{m}$. From Lemma 2 from \cite{Du1} it follows that
$P_\mu$ and $\Delta^m_\mu$ are homoclinically related. Therefore
there exists a basic set $\widetilde{\Delta}_\mu^m$ such that
$P_\mu\in \widetilde{\Delta}_\mu^m$ and $\Delta_\mu^m \subset
\widetilde{\Delta}_\mu^m$. In particular, for $\mu\in
\mathcal{R}_2=\cap_{m\ge 1} \mathcal{A}_m$ we have $\text{\rm dim}_HH(P_\mu,
f_\mu)=2.$ Set $\mathcal{R}=\mathcal{R}_1\cap \mathcal{R}_2$. This
proves (3.2).

The last property (3.3) follows from the following Lemma.
\blm\label{l.dimdense1} Let $\Lambda\subset M^2$ be a basic set of a surface
diffeomorphism. Then $$ \text{\rm dim}_H\{x\in \Lambda |\
\mathcal{O}^+(x) \text{\rm\ is dense in  } \Lambda \text{\rm\ and }
\mathcal{O}^-(x) \text{\rm\ is dense in  } \Lambda\}=\text{\rm
dim}_H\Lambda.
$$ \elm
Indeed, take any $\mu\in \mathcal{R}$, and consider the set $S_{\mu}=\{x\in H(P_{\mu}, f_{\mu})\ |\ P_{\mu}\in \omega(x)\cap \alpha(x)\}$. By construction of $\mathcal{R}$, for any $m\in \mathbb{N}$ the homoclinic class $H(P_{\mu}, f_{\mu})$ contains a hyperbolic set $\widetilde{\Delta}_\mu^m$ such that $P_\mu\in \widetilde{\Delta}_\mu^m$ and $\text{\rm dim}_H\widetilde{\Delta}^m_\mu
> 2-\frac{1}{m}$. Due to Lemma \ref{l.dimdense1}, the set of points in $\widetilde{\Delta}_\mu^m$ whose positive and negative semiorbits are both dense in $\widetilde{\Delta}_\mu^m$ has Hausdorff dimension greater than  $2-\frac{1}{m}$, and, hence, $\text{\rm dim}_H S_{\mu}\ > 2-\frac{1}{m}$. Since this is true for arbitrary large $m$, we have $\text{\rm dim}_H S_{\mu}= 2$.

 Lemma \ref{l.dimdense1} is a folklore, and can be seen as a corollary of results from \cite{Ma, MM}.

This completes the proof of Theorem \ref{t.maintheorem}.
\end{proof}

\section{On Hausdorff dimension of stochastic sea of the standard map: the proof}\label{s.standardproof}

In this section we derive the results on the Hausdorff dimension of stochastic sea of the standard map (Theorems \ref{t.1} and \ref{t.2}) from the result on hyperbolic sets of large Hausdorff dimension appearing after a conservative homoclinic bifurcation (Theorem \ref{t.maintheorem}).

 In the study of the standard family in the
current context Duarte \cite{Du3} proved the following important
results:

\noindent {\bf Theorem A (Duarte, \cite{Du3}).} {\it There is a
family of basic sets $\Lambda_k$ of $f_k$ such that:

1. $\Lambda_k$ is dynamically increasing, meaning for small
$\varepsilon >0$, $\Lambda_{k+\varepsilon}$ contains the
continuation of $\Lambda_k$ at parameter $k+\varepsilon$.


2. Hausdorff Dimension of $\Lambda_k$ increases up to 2. For large
$k$,
$$
\text{\rm dim}_H(\Lambda_k)\ge 2\frac{\log 2}{\log \left(2+\frac{9}{k^{1/3}}\right)}.
$$


3. $\Lambda_k$ fills in $\T^2\ni (x,y)$, meaning that as $k$ goes to
$\infty$ the maximum distance of any point in $\mathbb{T}^2$ to
$\Lambda_k$ tends to $0$. For large $k$, the set $\Lb_k$ is
$\dt_k$-dense on $\T^2$ for $\delta_k=\frac{4}{k^{1/3}}$. }

\noindent {\bf Theorem B (Duarte, \cite{Du3}).} {\it There exists
$k_0>0$ and a residual set $R\subseteq [k_0, \infty)$ such that for
$k\in R$ the closure of the $f_k$'s elliptic points contains
$\Lambda_k$. }

\noindent {\bf Theorem C (Duarte, \cite{Du3}).} {\it There exists
$k_0>0$ such that given any $k\ge k_0$ and any periodic point $P\in
\Lambda_k$, the set of parameters $k'\ge k$ at which the invariant
manifolds $W^s(P(k'))$\footnote{Recall that $P(k')$ denotes the continuation of the periodic saddle
$P$ at parameter $k'$.} and $W^u(P(k'))$ generically unfold a
quadratic tangency is dense in $[k, +\infty).$}

Theorem \ref{t.1} should be considered as an improvement of Theorems A and B.

\begin{proof}[Proof of Theorem \ref{t.1} and Theorem \ref{t.2}]
 We begin with the following technical statement. Denote by $\cN(N)=(n_1, \ldots ,n_N)$ an $N$-tuple with
$n_i\in \mathbb{N}$.
\bprop\label{p.conditions} There exists $k_0>0$ such that for each
$N\in \mathbb{N}$ there is a family of finite open intervals $
\mathcal{U}_{\cN(N)}\subseteq [k_0, +\infty) $ indexed by $N$-tuples
$\cN(N)=(n_1, \ldots ,n_N)$ satisfying the following properties:
\begin{description}
\item[\bf U1] For pair of tuples $\cN(N)\ne \cN'(N)$ intervals $\
\mathcal{U}_{\cN(N)}$ and $\mathcal{U}_{\cN'(N)}$ are disjoint.
\item[\bf U2] For  any tuple $\cN(N+1)=(\cN(N), n_{N+1})$ we have $
\mathcal{U}_{\cN(N+1)}\subseteq \mathcal{U}_{\cN(N)}. $
\item[\bf U3] The union $\cup_{n_1\in }\mathcal{U}_{n_1}$ is dense in $[k_0,
+\infty)$, and for each $N\in \mathbb{N}$ the union $\cup_{j\in
\mathbb{N}}\, \mathcal{U}_{(\cN(N),j)}$ is dense in
$\mathcal{U}_{\cN(N)}$.
\item[\bf U4]
Every diffeomorphism $f_k$, $k\in \mathcal{U}_{\cN(N)}$, has a
sequence of invariant basic sets
$$
\Lambda_k^{(n_1)}\subseteq \Lambda_k^{(n_1, n_2)}\subseteq \ldots
\subseteq \Lambda_k^{\cN(N)},
$$
and $\Lambda_k^{\cN(N)}$ depends continuously on $k\in
\mathcal{U}_{\cN(N)}$.
\item[\bf U5]
$\Lambda_k\subseteq \Lambda_k^{(n_1)}$ for each $n_1\in \mathbb{N}$
and  $k\in \mathcal{U}_{n_1}$, where $\Lambda_k$ is a hyperbolic set
from Theorem A.
\item[\bf U6]
$\text{\rm dim}_H\Lambda_k^{\cN(N)}>2-1/N$.
\item[\bf U7] For any point $x\in \Lambda_k^{\cN(N)}$ there exists an elliptic
periodic point $p_x$ of $f_k$ such that $\text{dist}(p_x, x)<1/N$.
\end{description}
\enprop
\begin{proof}[Proof of Proposition \ref{p.conditions}]
Notice that Theorem \ref{t.maintheorem} and Theorem C directly imply the following statement:

\blm\label{l.reduction} Given $k^*\in (k_0, +\infty)$,
$\varepsilon>0$ and $\delta>0$, there exists  a finite open interval
$V\subset (k^*-\varepsilon, k^*)$ such that for all $k\in V$ the map
$f_k$ has a basic set $\Lambda^*_k$ such that

1) $\Lambda^*_k$ depends continuously on $k\in V$;

2) $\Lambda^*_k\supseteq \Lambda_k,$ where $\Lambda_k$ is a basic set from Theorem A;

3) Hausdorff dimension $\text{\rm dim}_H\Lambda_k^{*}>2-\delta$,

4) For any point $x\in \Lambda_k^{*}$ there exists an elliptic
periodic point $p_x$ of $f_k$ such that $\text{dist}(p_x,
x)<\delta$. \elm

Proposition
\ref{p.conditions} can be reduced to Lemma \ref{l.reduction}.  Indeed, let us show how to construct the intervals
$\mathcal{U}_{n_1}$ and  the sets $\Lambda_k^{(n_1)}$. Let
$\{k_l\}_{l\in \mathbb{N}}$ be a dense set of points in $(k_0,
+\infty)$. Apply Lemma \ref{l.reduction} to each $k^*=k_l$, $l\in
\mathbb{N}$, for $\delta=1$,
$\varepsilon=\varepsilon_l<\frac{1}{l}$. That gives a sequence of
open intervals $\{V_l\}_{l\in \mathbb{N}}$. Since the sequence
$\{k_l\}_{l\in \mathbb{N}}$ is dense in $(k_0, +\infty)$ and
$\varepsilon_l\to 0$, intervals $\{V_l\}$ are dense in $(k_0,
+\infty)$.

Take $\mathcal{U}_1=V_1$. If $\mathcal{U}_1, \ldots, \mathcal{U}_t$
are constructed, take $V_s$ -- the first interval in the sequence
$\{V_l\}_{l\in \mathbb{N}}$ that is not contained in
$\overline{\cup_{n_1=1}^{t}\mathcal{U}_{n_1}}$. Then $V_s\backslash
\overline{\cup_{n_1=1}^{t}\mathcal{U}_{n_1}}$ is a finite union of
$K$ open intervals. Take those intervals as $\mathcal{U}_{t+1},
\ldots, \mathcal{U}_{t+K}$, and continue in the same way. This gives
a sequence of a disjoint intervals $\{\mathcal{U}_{n_1}\}_{n_1\in
\mathbb{N}}$ with desired properties.

Now, assume that intervals  $\{\mathcal{U}_{\cN(N)}\}$ are
constructed. Take one of the intervals $\mathcal{U}_{\cN(N)}$. Due to Theorem C, the
set $\Lambda_k^{\cN(N)}$ exhibits persistent tangencies. Therefore, application of Theorem \ref{t.maintheorem} gives  a dense
sequence of intervals $\{V_{\cN(N), l}\}_{l\in \mathbb{N}}$ in
$\mathcal{U}_{\cN(N)}$ such that for each $k\in V_{\cN(N), l}$ the
map $f_k$ has a basic set $\Delta_k$ such that Hausdorff dimension
$\text{\rm dim}_H\Delta_k>2-{\tiny \dfrac{1}{N+1}}$ and
$\Delta_k\cap \Lambda_k^{\cN(N)}\ne \emptyset$.


Now we need the following lemma from hyperbolic
dynamics.
 \blm\label{l.basicsets}
Let $\Delta_1$ and $\Delta_2$ be two basic sets (i.e. locally maximal transitive hyperbolic sets) 
of a diffeomorphism
$f:M^2\to M^2$ of a surface $M^2$ that are homeomorphic to a Cantor set. Suppose that $\Delta_1\cap \Delta_2\ne \emptyset$.   Then there is a
basic set $\Delta_3\subseteq M^2$ such that $\Delta_1\cup
\Delta_2\subseteq \Delta_3$.
 \elm

 \brm
 Having in mind some possible generalizations, we notice that Lemma \ref{l.basicsets} holds also for higher dimensional diffeomorphisms (two-dimensionality of the phase space is not used in the proof).
 \erm

\begin{proof}[Proof of Lemma \ref{l.basicsets}] Due to recent result of Anosov \cite{An} any zero-dimensional hyperbolic set is contained in a locally maximal hyperbolic set. Therefore in our case $\Delta_1\cup \Delta_2$ is contained in some locally maximal hyperbolic set $\widetilde{\Delta}$. Spectral decomposition theorem claims that $\widetilde{\Delta}$ is a finite disjoint union of basic sets. One of these basic sets must contain $\Delta_1$, and since $\Delta_1\cap \Delta_2\ne \emptyset$, the same basic set has to contain $\Delta_2$.
\end{proof}

Apply Lemma \ref{l.basicsets} to $\Delta_k$ and
$\Lambda_k^{\cN(N)}$, and denote by
$\widetilde{\Lambda}_k^{\cN(N)}\supset \Delta_k\cup
\Lambda_k^{\cN(N)}$ the corresponding basic set. The set
$\widetilde{\Lambda}_k^{\cN(N)}$ also has persistent tangencies. The
unfolding of a homoclinic tangency creates elliptic periodic orbits
which shadow the orbit of homoclinic tangencies. The creation of
these generic elliptic points can be seen from the renormalization
at conservative homoclinic tangencies, see \cite{MR1}. Shrinking
$V_{(\cN(N), l)}$ if necessary we can guarantee that
$\widetilde{\Lambda}_k^{\cN(N)}$ can be $\frac{1}{N+1}$-accumulated
by elliptic periodic points. Now the same procedure that we applied
above to intervals $\{V_l\}$ gives a collection of disjoint
intervals $\{\mathcal{U}_{(\cN(N), n_{N+1})}\}_{n_{N+1}\in
\mathbb{N}}$ in $\mathcal{U}_{\cN(N)}$. For any $k\in
\mathcal{U}_{\cN(N), n_{N+1}}\subset V_{(\cN(N), l)}$ we take
$\Lambda_k^{(\cN(N), n_{N+1})}=\widetilde{\Lambda}_k^{\cN(N)}$. By construction,
all the properties in Proposition \ref{p.conditions} are now satisfied.
\end{proof}

Now let us explain how Theorems \ref{t.1}  follows from Proposition
\ref{p.conditions}. Set $\mathbf{U}_N=\cup_{\cN(N)}
\mathcal{U}_{\cN(N)}$. Due to {\bf U3)} the set $\mathbf{U}_N$ is
dense in $[k_0, +\infty)$. Therefore $\mathcal{R}=\cap_{N\in
\mathbb{N}}\mathbf{U}_N$ is a residual subset of $[k_0, +\infty)$.
Properties {\bf U1)} and {\bf U2)} imply that for each $k\in
\mathcal{R}$ the value $k$ belongs to each element of the uniquely
defined nested sequence of intervals
$$
\mathcal{U}_{n_1}\supseteq \mathcal{U}_{n_1, n_2}\supseteq \ldots
\supseteq \mathcal{U}_{\cN(N)} \supseteq \ldots
$$
Therefore for $k\in \mathcal{R}$ the sequence of basic sets
$$
\Lambda_k\subseteq \Lambda_k^{(n_1)}\subseteq \Lambda_k^{(n_1,
n_2)}\subseteq \ldots \subseteq \Lambda_k^{\cN(N)}\subseteq \ldots
$$
is defined such that Hausdorff dimension $\text{\rm dim}_H
\Lambda_k^{\cN(N)}>2-1/N$. Since $k$ is fixed now, redenote
$\Lambda_k^N=\Lambda_k^{\cN(N)}$. Items 1.-- 3. of Theorem
\ref{t.1} follows from {\bf U5)} and {\bf U6)}.

The closure of the union of a nested sequence of transitive sets is
transitive, so property 4. follows.

 For a locally maximal transitive invariant hyperbolic set of a surface
diffeomorphism the Hausdorff dimension of the set is equal to the
Hausdorff dimension of any open subset of this set, see \cite{MM}.
This implies the property 5.  for the sets $\Omega_k,
k\in \mathcal{R}$.

Property 6. follows directly from {\bf U7)}.

Finally, in order to prove Theorem \ref{t.2} it is enough to consider the
family of basic sets $\Lambda_k^{\cN(N)}$ defined for $k\in
\mathbf{U}_N$ for large enough $N$.
\end{proof}

\end{document}